%% file: article_main.tex
\documentclass[onefignum,onetabnum]{siamart190516}

\input{shared}

\ifpdf
\hypersetup{
  pdftitle={Convexification with bounded gap for randomly projected QO},
  pdfauthor={T. Fuji, P. L. Poirion, and A. Takeda}
}
\fi

\begin{document}

\maketitle

\begin{abstract}
  Random projection techniques based on Johnson-Lindenstrauss lemma are used for randomly aggregating the constraints or variables of optimization problems while approximately preserving their optimal values, that leads to smaller-scale optimization problems.
  D’Ambrosio et al. have applied random projection to a quadratic optimization problem so as to decrease the number of decision variables.
  Although the problem size becomes smaller, the projected problem will also almost surely be non-convex if 
  the original problem is non-convex, and hence will be hard to solve.

  In this paper, by focusing on the fact that the level of the  non-convexity of a non-convex quadratic optimization problem can be alleviated 
  by random projection, we find an approximate global optimal value of the problem by attributing it to a convex problem with smaller size. 
  To the best of our knowledge, our paper is the first to use random projection for convexification of non-convex optimization problems.
  We evaluate the approximation error between  optimum values of a non-convex optimization problem and its convexified randomly projected problem.

\end{abstract}



\input{./introduction.tex}

\input{./preliminaries.tex}

\input{./crp.tex}

\input{./scaling_preconditioning.tex}

\input{./numerical_experiments.tex}

\input{./conclusions.tex}



\bibliographystyle{siamplain}
\bibliography{ref}
\end{document}

%% file: shared.tex
\usepackage{lipsum}
\usepackage{amsfonts}
\usepackage{graphicx}
\usepackage{epstopdf}
\usepackage{algorithmic}
\ifpdf
  \DeclareGraphicsExtensions{.eps,.pdf,.png,.jpg}
\else
  \DeclareGraphicsExtensions{.eps}
\fi

\newsiamremark{remark}{Remark}
\newsiamremark{hypothesis}{Hypothesis}
\crefname{hypothesis}{Hypothesis}{Hypotheses}
\newsiamthm{claim}{Claim}
\newsiamthm{example}{Example}
\newsiamthm{assumption}{Assumption}

\headers{Convexification with bounded gap for randomly projected QO}{T. Fuji, P. L. Poirion, and A. Takeda}

\title{Convexification with bounded gap for randomly projected quadratic optimization
\thanks{To be submitted.
\funding{This work was supported by JSPS KAKENHI (17H01699 and 19H04069) and JSTERATO (JPMJER1903).}}}

\author{Terunari Fuji\thanks{Graduate School of Information Science and Technology, The University of Tokyo, Tokyo 113-8656, Japan.
  (\email{teru0818258@g.ecc.u-tokyo.ac.jp}).}
\and Pierre-Louis Poirion\thanks{Center for Advanced Intelligence Project, RIKEN, Tokyo 103-0027, Japan.
  (\email{pierre-louis.poirion@riken.jp}).}
\and Akiko Takeda\thanks{Graduate School of Information Science and Technology, The University of Tokyo, Tokyo 113-8656, Japan and
Center for Advanced Intelligence Project, RIKEN, Tokyo 103-0027, Japan.
  (\email{takeda@mist.i.u-tokyo.ac.jp}).}
}

\usepackage{amsopn}

\usepackage{color}
\usepackage{amsmath,amssymb,bm}
\usepackage{array,arydshln}

\usepackage{comment}

\usepackage{breqn}

\usepackage{subfigure}

\newcommand{\ds}{\displaystyle}
\newcommand{\eps}{\varepsilon}
\newcommand{\R}{\mathbb R}

\newcommand{\norm}[1]{\left\|#1\right\|}
\newcommand{\tliner}[2]{\langle #1, #2\rangle}

\newcommand{\T}{{\mathsf T}}
\newcommand{\F}{{\mathcal F}}
\newcommand{\C}{{\mathcal C}}

\DeclareMathOperator{\diag}{diag}
\DeclareMathOperator{\Prob}{Prob}
\DeclareMathOperator{\tr}{tr}

\newcommand{\cond}{\mathrm{cond}}

\newcommand{\opt}[1]{\mathrm{opt}(#1)}
\renewcommand{\P}{{\mathbf P}}
\newcommand{\RP}{\mathbf{RP}}
\newcommand{\CRP}{\mathbf{CRP}}

\usepackage{ifthen}
\newcommand{\COMM}[2]{{
\ifthenelse{\equal{#1}{TF}}{\color{red}}{
\ifthenelse{\equal{#1}{PL}}{\color{blue}}{
\ifthenelse{\equal{#1}{AT}}{\color{magenta}}}}[#1: #2]}}

%% file: introduction.tex
\section{Introduction}\label{sec:introduction}


We consider the following  non-convex quadratic optimization problem having large-scale decision variables $x\in\R^n$:
\begin{equation}\label{eq:problem_P}
\P \equiv \min_x \{x^\T Qx + c^\T x\mid Ax \le b\},
\end{equation}
where $Q\in\R^{n\times n}$ is a symmetric matrix, $A\in\R^{m\times n}, c\in\R^n$ and $b\in\R^m$.
Here, we assume that the feasible region is full dimensional (hence, equality constraints are excluded) and 
at least one of the eigenvalues of $Q$ is positive.

In this paper, we find a feasible solution with a bounded approximation error to the optimal value of \cref{eq:problem_P}.
For the purpose, we reduce the non-convex problem
to a lower-dimensional convex optimization problem using random projection and convexification techniques,
and evaluate the gap between optimum values of the two optimization problems.
Random projection 
refers to the technique that maps a set of points $X \subsetneq \mathbb{R}^n$ to a set $PX  \subsetneq \mathbb{R}^d $ in a lower dimensional subspace
with random matrices $P \in \mathbb{R}^{d \times n}$ 
in a way that some intrinsic properties of the set $X$ are approximately preserved with high probability. 
The main idea of random projections comes from the Johnson-Lindenstrauss lemma (\cite{jllemma})
that states that if the probability distribution of $P$ is properly chosen then there exists $d<n$ such that the Euclidean distance between any pair $x,y \in X $ of points in $X$ is approximately preserved with high probability, i.e. $\|Px-Py\| \approx \|x-y\|.$

Random projections have been already used in various studies to reduce the size of an input matrix while retaining most of its information (see for example \cite{woodruff2014, drineas2016, derezinski202}),
and they are often used for some machine learning problems.
Notice that this framework may also be referred as \emph{sketching}, if random projection matrices are used not to reduce the dimension of the decision variables $x$ 
but to reduce 
the sample size of the data
in the problem. For example in a least-square problem setting, i.e. $\min\limits_{x \in C} \|y-X x\|^2$ where $C\subseteq \mathbb{R}^n$ is a convex set, $X \in \mathbb{R}^{m'\times n}$ is the design matrix and $y \in \mathbb{R}^{m'}$ is the response vector,
the sample size $m'$ is reduced using random projections. 
The use of random projections to approximate least-squares problems has been extensively studied by, for example, \cite{Mahoney2011, chen2020, pilanci2014, wang2017}.
Random projections have also been used in a non-convex setting: in \cite{NIPS2010}, the authors apply random projections for the $k$-means clustering problem to reduce 
the number of data points, i.e., the sample size in the problem. Since the reduced problem is also a non-convex optimization problem, the error from its optimum value is evaluated under the assumption that an approximation algorithm is used.


While random projection has also been applied to reduce the number of constraints of a Linear Problem (LP) by randomly aggregating them in  \cite{MOR_RP}, it has been applied to a Quadratic Problem (QP) so as to decrease the number of decision variables in  \cite{d2019random}.
In \cite{bluhm2019} the authors apply random projections to Semi-Definite Programming (SDP):  the variables of the SDP are randomly projected to a space of lower dimension.



In this paper, we show that random projections can also be applied to convexify a non-convex optimization problem. More precisely, we will use random projection to define a convexification of \cref{eq:problem_RP} and give some error bounds for the error between these two problems. Notice that
in \cite{d2019random} the authors already use a random projection matrix $P\in\R^{d\times n}$ to project \cref{eq:problem_P} into the following QP:
\begin{equation}\label{eq:problem_RP}
\RP \equiv \min_u \{u^\T \bar Q u + \bar c^\T u\mid \bar A u \le b\},
\end{equation}
where $u\in\R^d$, $\bar Q = PQP^\T, \bar c = Pc$ and $\bar A = AP^\T$.
However, although \cref{eq:problem_RP} is a QP of smaller size, i.e. the variables of \cref{eq:problem_RP} belong to a smaller dimensional space, if \cref{eq:problem_P} is non-convex then the projected problem will also almost surely be non-convex, and hence will be hard to solve.
Therefore, we focus on the fact that if $d$ is small enough then eigenvalues of the matrix $\bar{Q}$ are skewed towards positive values (see \cref{fig:eigenvalues_Q_PQP} shown later), which implies that ignoring negative eigenvalues for the reduced matrix due to the convexification does not lose much information about problem \cref{eq:problem_P}.
In this paper, taking advantage of the fact, 
we show the following: 
if the dimension $d$ is carefully chosen then \cref{eq:problem_P} can be approximated by a convex QP of smaller size. More precisely, we consider the following convex QP: 
\begin{equation}\label{eq:problem_CRP}
\CRP \equiv \min_u \{u^\T \bar Q^+ u + \bar c^\T u\mid \bar A u \le b\},
\end{equation} 
where $\bar Q^+ = \F^+(\bar Q)$ is the projection of $\bar Q$ onto the positive semidefinite cone.
Using an optimum $u^*$ of \cref{eq:problem_CRP}, we have a feasible solution $P^\top u^*$ to \cref{eq:problem_P},
for which an approximation error from the optimum value of \cref{eq:problem_P} is estimated.

To the best of our knowledge, our paper is the first to use random projection for convexification of non-convex optimization problems.
We evaluate the approximation error between  optimum values of a non-convex optimization problem and its convexified problem. 
More precisely, we will prove that if the dimension $d$ is properly chosen then the optimal value of $\CRP$ is a good approximation of the one of $\P$. 


The rest of this paper is organized as follows. 
In \cref{sec:preliminaries}, we introduce mathematical preliminary.
In \cref{sec:CRP}, we prove our main results on approximate optimality under the assumption that $\tr Q >0$ and in \cref{sec:scaling_preconditioning}, we discuss how to relax
the assumption while achieving similar theoretical results.
In \cref{sec:numerical_experiments} we discuss the results of numerical experiments for two types of problems: randomly generated problems and support vector machine (SVM) problems with indefinite kernel 
which are attributed to non-convex quadratic optimization problems. Conclusions follow in \cref{sec:conclusions}.

All the notations used in this paper are in \cref{Table:notations}.

\begin{table}[htb]
  \centering
  \caption{Notations in this paper.}
    \begin{tabular}{|c|l|} \hline
      Notation & Convention \\ \hline \hline
      $\C_0, \C_1, \C_2, \C_3$ & absolute constant \\ \hline
      $\C$ & $\C = \max\{\C_2, \C_3\}$ \\ \hline
      $\norm{X}_{\psi_2}$ & the sub-Gaussian norm of a sub-Gaussian random variable\\ \hline
      $\norm{X}_{\psi_1}$ & the sub-exponential norm of a sub-exponential random variable\\ \hline
      $\norm{a}$ & the Euclidean norm of a vector $a$\\ \hline
      $\norm{M}$ & the operator norm of a matrix $M$ : $\norm{M} = \max_{\norm{x}=1} \norm{Mx}$\\ \hline
      $\norm{M}_F$ & the Frobenius norm of a matrix $M$ : $\norm{M}_F = \sum_{ij} M_{ij}^2 $\\ \hline
      $\F^+(M)$ & the projection onto the positive semidefinite cone of a matrix $M$\\ \hline
      $\bm{1}$     & the all one vector \\ \hline
      $I_n$     & the identity matrix of size $n$\\ \hline
      $\diag(a)$ & the matrix whose diagonal is the vector $a$ \\ \hline
      $\cond(M)$ & the condition number of a matrix $M$ \\ \hline
      $\opt{\mathbf F}$ & the optimal value of an optimization problem $\mathbf F$ \\ \hline
      $\mathbb{E}(X)$ & expectation of a random variable $X$ \\ \hline
      $\delta_{ij}$ & Kronecker delta: $\delta_{ij}=1$ if $i=j$, $\delta_{ij}=0$ otherwise \\ \hline
      ${\rm N}(\mu, \Sigma)$ & the normal distribution with mean $\mu$ and covariance $\Sigma$ \\ \hline
    \end{tabular}
    \label{Table:notations}
\end{table}

%% file: preliminaries.tex
\section{Preliminaries}\label{sec:preliminaries}

\subsection{Sub-Gaussian and sub-exponential random variables}
In this section we review some necessary definitions and theorems  
in the paper.
First, we recall some properties of sub-Gaussian and sub-exponential random variables and concentration inequalities.

\begin{definition}[Sub-Gaussian random variables]\label{def:sub-Gaussian}
  A random variable $X$ that satisfies 
  $$\mathbb{E}[\exp(X^2/K^2)] \le 2$$
  for some $K>0$ is called a sub-Gaussian random variable.
  The sub-Gaussian norm of $X$, denoted $\norm{X}_{\psi_2}$, 
  is defined to be the smallest $K$ that satisfies the above inequality, 
  or equivalently, we define 
  $$\norm{X}_{\psi_2} = \inf\{s > 0 \mid \mathbb{E}[\exp(X^2/s^2)] \le 2\}.$$
\end{definition}

\begin{lemma}[{\cite[Example 2.5.8]{vershynin2018high}}] \label{remark:gaussianissubgaussian}
  A Gaussian random variable $X\sim \mathrm N(0,\sigma^2)$ is sub-Gaussian with 
  $\norm{X}_{\psi_2} \le \C_2 \sigma$,
  where $\C_2$ is an absolute constant. 
\end{lemma}


\begin{definition}[Sub-exponential random variables]\label{def:sub-exponential}
  A random variable $X$ that satisfies 
  $$\mathbb{E}[\exp(|X|/K)] \le 2$$
  for some $K>0$ is called a sub-exponential random variable.
  The sub-exponential norm of $X$, denoted $\norm{X}_{\psi_1}$, 
  is defined to be the smallest $K$ that satisfies the above inequality, 
  or equivalently, we define 
  $$\norm{X}_{\psi_1} = \inf\{s > 0 \mid \mathbb{E}[\exp(|X|/s)] \le 2\}.$$
\end{definition}

\begin{lemma}[{\cite[Exercise 2.7.10]{vershynin2018high}}]\label{lem:subexponential_centering}
  For a sub-exponential random variable $Z$, then $Z-\mathbb{E}[Z]$ is sub-exponential too, and 
  $$\norm{Z-\mathbb{E}[Z]}_{\psi_1} \le \C_3 \norm{Z}_{\psi_1},$$
  where $\C_3$ is an absolute constant. 
\end{lemma}

In the following, we often use the absolute constant $\C$
 defined by $\C = \max (\C_2, \C_3)$.

Sub-Gaussian and sub-exponential distributions are closely related 
as we can see in the following lemma,
which implies that the product of sub-Gaussian random variables is sub-exponential.

\begin{lemma}[{\cite[Lemma 2.7.7]{vershynin2018high}}]\label{lem:productof2subexponential}
  Let $X$ and $Y$ be sub-Gaussian random variables, 
  then $XY$ is sub-exponential. Moreover,
  $$\norm{XY}_{\psi_{1}} \leq\norm{X}_{\psi_{2}}\norm{Y}_{\psi_{2}}.$$
\end{lemma}

We also recall Bernstein's inequality for sub-exponential random variables.

\begin{theorem}[Bernstein's inequality, {\cite[Theorem 2.8.1]{vershynin2018high}}] \label{thm:bernstein_ineq}
  Let $X_1,X_2,\dots,X_N$ be independent, mean zero, sub-exponential random variables.
  Then, for every $t\ge 0$, we have
  $$\Prob\left(\left|\sum_{i=1}^{N} X_{i}\right| \geq t\right) \leq 2 \exp \left(-\C_{1} \min \left(\frac{t^{2}}{\sum_{i=1}^N\left\|X_{i}\right\|_{\psi_{1}}^{2}}, \frac{t}{\max_i \left\|X_{i}\right\|_{\psi_{1}}}\right)\right),$$
  where $\C_1$ is an absolute constant.
\end{theorem}

\subsection{Definitions of \texorpdfstring{$\eps$}--net and estimation of the operator norm of a matrix}
Next, we recall the definition of $\eps$-net.

\begin{definition}\label{def:eps-net}
  Consider a subset $K \subset \R^n$ and let $\eps>0$. 
  A subset $\mathcal N\subseteq K$ is called an $\eps$-net of $K$ 
  if every point in $K$ is within distance $\eps$ of some point of $\mathcal N$, i.e.
  $$^\forall x \in K, ^\exists y \in \mathcal N, \norm{x-y} \le \eps.$$
\end{definition}

\begin{lemma}[{\cite[Corollary 4.2.13]{vershynin2018high}}]\label{remark:exists_eps_net}
  There exists a $\eps$-net with size $\ds\left(\frac 2 \eps +1\right)^n$
  of the unit $n$-Euclidean ball. 
\end{lemma}

$\eps$-nets can help us estimate the operator norm of a matrix.

\begin{lemma}[{\cite[Lemma 4.4.1, Exercise 4.4.3]{vershynin2018high}}]\label{lem:eps_net2operator_norm}
  Let $A$ be an $m\times n$ matrix and $\eps \in [0, 1)$. 
  Then, for any $\eps$-net $\mathcal N$ of the unit sphere $S^{n-1}$, we have
  $$\sup_{x\in\mathcal N} \norm{Ax} \le \norm{A} \le \frac{1}{1-\eps}\cdot\sup_{x\in\mathcal N} \norm{Ax}.$$

  Moreover, if $m=n$ and $A$ is symmetric, we have
  $$\sup_{x\in\mathcal N} |x^\T Ax| \le \norm{A} \le \frac{1}{1-2\eps}\cdot\sup_{x\in\mathcal N} |x^\T Ax|.$$
\end{lemma}

\subsection{Properties of random projections}
Now we recall basic properties of random projection matrices.
In this paper we call a matrix $P \in \R^{d\times n}$ a random projection matrix or a random matrix when its entries $P_{ij}$ are independently sampled from $\mathrm N (0, 1/d)$.

One of the most important features of a random projection defined by a random matrix 
is that it nearly
preserves the norm of any given vector with arbitrary high probability.
The following lemma is known as a variant of the Johnson-Lindenstrauss lemma  (\cite{jllemma}).

\begin{lemma}[{\cite[Lemma 5.3.2, Exercise 5.3.3]{vershynin2018high}}] 
  \label{lem:JLL_random_matrix}
  Let $P\in \R^{d\times n}$ be a random matrix
  whose entries $P_{ij}$ are independently drawn from $\mathrm N(0,1/d)$.

  Then for any $x\in\R^n$ and $\eps\in(0,1)$, we have
  $${\rm Prob\ }[(1-\eps)\norm{x}^2 \le \norm{Px}^2 \le (1+\eps)\norm{x}^2]
    \ge 1-2\exp(-\mathcal C_0 \eps^2 d), $$
  where $\mathcal C_0$ is an absolute constant.
\end{lemma}

Random projections also approximately preserve inner products, linear function values and quadratic function values.

\begin{lemma}[{\cite[Lemma 3.1, 3.2, 3.3]{d2019random}}]\label{lem:basic_properties_of_random_matrix}
  Let $P\in \R^{d\times n}$ be a random matrix
  whose entries $P_{ij}$ are independently sampled from $\mathrm N(0,1/d)$.
  Then for any $x,y\in\R^n$, $A\in \R^{m\times n}$ having unit row vectors, $Q\in \R^{n\times n}$ and $\eps\in(0,1)$, 
  the following probabilistic inequalities hold.
  \begin{enumerate}
    \renewcommand{\labelenumi}{(\roman{enumi})}
    \item With probability at least $1-4\exp(-\C_0 \eps^2 d)$, we have
    $$x^\T y -\eps\norm{x}\norm{y}  \le x^\T P^\T Py \le x^\T y + \eps\norm{x}\norm{y}.$$

    \item With probability at least $1-4m\exp(-\C_0 \eps^2 d)$,
    $$Ax-\eps \norm{x} \bm 1 \le  AP^\T P x \le Ax+\eps \norm{x} \bm 1.$$

    \item With probability at least $1-8\operatorname{rank} Q \cdot\exp(-\C_0 \eps^2 d)$,
    $$x^\T Qx - 3\eps \norm{x}^2 \norm{Q}_F \le 
      x^\T P^\T PQP^\T Px  \le x^\T Qx + 3\eps \norm{x}^2 \norm{Q}_F.$$
  \end{enumerate}
\end{lemma}

The above lemma, which estimates the error induced by random projections on different values, 
will be used to bound the error on the optimal value of a randomly projected quadratic optimization problem.

%% file: crp.tex
\section{Convexified randomly projected problem}\label{sec:CRP}
\subsection{Convexifying the objective function}
In this section we give an error bound between $\P$ and $\CRP$.
First we consider the distribution of the eigenvalues of $\bar Q = PQP^\T$. 
We can easily confirm that $\mathbb{E}[PQP ^\T] = \frac{\tr Q}{d} I_d$:
\begin{align*}
  \mathbb{E}[(PQP ^\T)_{ij}] =& \mathbb{E}\left[\sum_{k}\sum_{l} P_{ik}Q_{kl}P_{jl}\right]\\
                     =& \sum_{k}\sum_{l} Q_{kl} \mathbb{E} \left[ P_{ik}P_{jl}\right]\\
                     =& \sum_{k}\sum_{l} Q_{kl} \frac{\delta_{ij}\delta_{kl}}{d}\\
                     =& \frac{\tr Q}{d} \delta_{ij},
\end{align*}
where $\delta_{ij}$ denotes the Kronecker delta symbol.
By the above equality, we expect the eigenvalues of $PQP^\T$ to be distributed around $\frac{\tr Q}{d}.$
One example of eigenvalue distributions of $Q$ and $PQP^\T$ is shown in \cref{fig:eigenvalues_Q_PQP}, 
where we observe that the eigenvalue distribution of $PQP^\T$ is skewed towards positive values and the negative spectrum of $PQP^\T$ is negligible. 
In the next lemma, we evaluate the maximum deviation between the eigenvalues of $PQP^\T$  and $\frac{\tr Q}{d}$.
\begin{figure}[htb]
  \begin{center} 
    \includegraphics[width=0.8\linewidth]{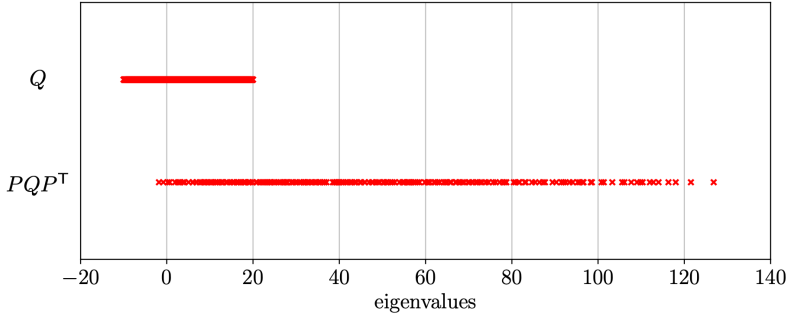}
     \caption{Distributions of eigenvalues of $Q$ and $PQP^\T$($Q\in\R^{2000\times 2000}$ with eigenvalues $\lambda_i(Q) = -10 + \frac{30i}{2000}$, $P\in\R^{200\times 2000}$).} 
    \label{fig:eigenvalues_Q_PQP} 
  \end{center}
\end{figure}

\begin{lemma}\label{lem:concentration_of_PQP}
  Let $Q$ be an $n\times n$ symmetric matrix and let
  $P$ be a $d\times n$ random matrix whose entries are sampled from $\mathrm N (0, 1/d)$.
  Then, for every $t\ge 0$, 
  $$\operatorname{Prob}\left[\norm{PQP^\T-\frac{\tr Q}{d} I_d} \ge t\right] 
  \le 2\cdot 9^d \exp \left(-\C_{1} \min \left(\frac{d^2t^2}{4\C^6 \norm{Q}_F^2}, \frac{dt}{2\C^3 \norm Q} \right)\right).$$
\end{lemma}

\begin{proof}
  First, using the eigenvalue decomposition of $Q$, we write
  $$Q = U \Lambda U^T,$$
  where $U$ is an orthogonal matrix and $\Lambda = \diag(\lambda_1,\dots, \lambda_n)$. Since the distribution of $PU$ is the same as the distribution of $P$, we have that the distribution of $PQP^\T$ is the same as the distribution of $P\Lambda P^\T$, hence
  $$\operatorname{Prob}\left[\norm{PQP^\T-\frac{\tr Q}{d} I_d} \ge t\right]
  =\operatorname{Prob}\left[\norm{P\Lambda P^\T-\frac{\tr Q}{d} I_d} \ge t\right].$$

  By  \cref{remark:exists_eps_net}, we can take a $1/4$-net $\mathcal N$ with a size of $9^n$
  of the unit $n$-Euclidean ball.
  Using  \cref{lem:eps_net2operator_norm} with $\eps=1/4$, we get
  \begin{align}
    &    \operatorname{Prob}\left[\norm{P\Lambda P^\T-\frac{\tr Q}{d} I_d} \ge t\right] \label{eq:concentration_of_PQP_eq1}\\
    \le& \operatorname{Prob}\left[2\sup_{x\in \mathcal N} \left|x^\T \left(P\Lambda P^\T-\frac{\tr Q}{d}I_d\right) x\right|  \ge t\right] \nonumber \\
    =&   \operatorname{Prob}\left[\sup_{x\in \mathcal N} \left|x^\T P\Lambda P^\T x - \frac{\tr Q}{d}\right|  \ge \frac t 2\right] \nonumber\\
    =&   \operatorname{Prob}\left[\sup_{x\in \mathcal N} \left|\sum_{j=1}^n\left(\lambda_j \tliner{P^j}{x}^2 - \frac{\lambda_j}{d}\right)\right|  \ge \frac t 2\right], \nonumber
  \end{align} 
  where the $P^j$ are the column vectors of $P$. 
  
  Let $X_j = \tliner{P^j}{x}$, then $X_j$ are independent Gaussian random variables of variances $1/d$.
  Thus, by \cref{remark:gaussianissubgaussian}, we obtain
  $$\norm{X_j}_{\psi_2} \le \C \sqrt\frac{1}{d}.$$
  With the above inequality,  \cref{lem:subexponential_centering} and  \cref{lem:productof2subexponential}, the random variable $\lambda_j X_j^2 - \dfrac{\lambda_j}{d}$ turns out to be sub-exponential whose sub-exponential norm is bounded by 
  $$\norm{\lambda_j X_j^2 - \dfrac{\lambda_j}{d}}_{\psi_1}
  \le \C\norm{\lambda_j X_j^2}_{\psi_1} 
  = \C|\lambda_j| \norm{X_j^2}_{\psi_1}
  \le \C|\lambda_j| \norm{X_j}_{\psi_2}^2
  \le \C^3 \frac{|\lambda_j|}{d}.$$
By Bernstein's inequality (\cref{thm:bernstein_ineq}),
  we obtain, for each $x\in\mathcal N$, the following inequality :
  \begin{align}
    &\operatorname{Prob}\left(\left|\sum_{j=1}^{n} \left(\lambda_j X_j^2 - \dfrac{\lambda_j}{d}\right)\right| \geq \frac t 2\right) \label{eq:concentration_of_PQP_eq2}\\
    \le& 2 \exp \left(-\C_{1} \min \left(\frac{\dfrac{t^2}{4}}{\sum_{j=1}^n \left(\C^3 \frac{|\lambda_j|}{d}\right)^2}, \frac{\dfrac{t}{2}}{\max_j \left(\C^3 \frac{|\lambda_j|}{d}\right)} \right)\right)\nonumber\\
    =& 2 \exp \left(-\C_{1} \min \left(\frac{d^2t^2}{4\C^6 \norm{Q}_F^2}, \frac{dt}{2\C^3 \norm Q} \right)\right). \nonumber
  \end{align}  
  Finally, 
  from \cref{eq:concentration_of_PQP_eq1}, \cref{eq:concentration_of_PQP_eq2} and a union bound on $\mathcal N$, we have 
  $$\Prob\left[\norm{PQP^\T-\frac{\tr Q}{d} I_d} \ge t\right] 
  \le 2\cdot 9^d \exp \left(-\C_{1} \min \left(\frac{d^2t^2}{4\C^6 \norm{Q}_F^2}, \frac{dt}{2\C^3 \norm Q} \right)\right),$$
  which completes the proof.
\end{proof}

\begin{corollary}\label{cor:neg_spec_of_PQP}
  Let $Q$ be an $n\times n$ symmetric matrix and let
  $P$ be a $d\times n$ random matrix whose entries are sampled from $\mathrm N (0, 1/d)$.
  If $\tr Q >0$, then for any $\eps>0$, we have
  \begin{align*}
    &\operatorname{Prob}\left[\left|\min(0, \lambda_{\min}(PQP^\T)) \right| \ge \eps \norm{Q}_F\right] \\
    &\qquad\le 
    2\cdot 9^d \exp \left(-\C_{1} \min \left(\frac{\left(\tr Q + \eps d\norm{Q}_F\right)^2}{4\C^6 \norm{Q}_F^2}, \frac{\tr Q + \eps d\norm{Q}_F}{2\C^3 \norm Q} \right)\right),
  \end{align*}
  where $\lambda_{\min}(M)$ denotes the minimum eigenvalue of a matrix $M$.
\end{corollary}

\begin{proof}
  Suppose that $\left|\min(0, \lambda_{\min}(PQP^\T))\right| \ge \eps\norm{Q}_F$. This implies that \\
  $\lambda_{\min}(PQP^\T)<0$ and $|\lambda_{\min}(PQP^\T)| \ge \eps \norm{Q}_F$. 
  Furthermore, since $\tr Q$ is positive, we have that
  $\lambda_{\min}(PQP^\T-\frac{\tr Q}{d} I_d)<0$ and  
  $|\lambda_{\min}(PQP^\T-\frac{\tr Q}{d} I_d)| \ge \frac{\tr Q}{d} + \eps \norm{Q}_F.$ 
  The last inequality implies that $\norm{PQP^\T-\frac{\tr Q}{d} I_d} \ge \frac{\tr Q}{d} + \eps \norm{Q}_F$.
  
  By the above argument, we obtain the following inequality: 
  \begin{align*}
    &\Prob\left[\left|\min(0, \lambda_{\min}(PQP^\T)) \right| \ge \eps \norm{Q}_F\right] \\
    &\qquad \le
    \Prob\left[\norm{PQP^\T-\frac{\tr Q}{d} I_d} \ge \frac{\tr Q}{d} + \eps \norm{Q}_F\right]. 
  \end{align*}
  
  Taking $t = \dfrac{\tr Q}{d} + \eps \norm{Q}_F$ in  \cref{lem:concentration_of_PQP} ends the proof.
\end{proof}
  
Next, we evaluate the difference between $x^\T Qx$ and $x^\T P^\T \bar Q^+ Px \\ (= x^\T P^\T \mathcal F^+\left(PQP^\T\right) Px)$ for a fixed vector $x$, where $\F^+$ denotes the projection onto the positive semidefinite cone.
The following theorem will be used to evaluate the error between the optimal values of $\P$ and $\CRP$.
In \cite{d2019random}, the authors use  \cref{lem:basic_properties_of_random_matrix} (iii) to evaluate the error between the optimal values of $\P$ and $\RP$. In this sense, \cref{thm:xQx_xPPQ_posPPx} is an extension of  \cref{lem:basic_properties_of_random_matrix} (iii).

The following theorem is proven by probabilistically evaluating the difference between the optimal values of $\RP$ and $\CRP$
 due to the randomness of problem $\RP$, though the convexification technique itself is a deterministic operation.

\begin{theorem}\label{thm:xQx_xPPQ_posPPx}
  Let $Q$ be an $n\times n$ symmetric matrix that satisfies $\tr Q >0$ and let
  $P$ be a $d\times n$ random matrix whose entries are sampled from $\mathrm N (0, 1/d)$.

  Then, for any $x\in\R^n$, with probability at least 
  \begin{align*}
  &1-8\operatorname{rank} Q \cdot\exp(-\C_0\eps_1^2 d)-2\exp(-\C_0\eps_2^2 d)\\
  &\qquad -2\cdot 9^d \exp \left(-\C_{1} \min \left(\frac{\left(\tr Q + \eps_3 d\norm{Q}_F\right)^2}{4\C^6 \norm{Q}_F^2}, \frac{\tr Q + \eps_3 d\norm{Q}_F}{2\C^3 \norm Q} \right)\right), 
  \end{align*}
  we have
  $$\left|x^\T Qx - x^\T P^\T \mathcal F^+\left(PQP^\T\right) Px\right| \le (3\eps_1 + \eps_3 + \eps_2\eps_3 )\norm{x}^2 \norm{Q}_F.$$
\end{theorem}
  
\begin{proof}
  By \cref{lem:JLL_random_matrix}, \cref{lem:basic_properties_of_random_matrix} (iii) and \cref{cor:neg_spec_of_PQP}, with probability at least 
  \begin{align*}
    &1-8\operatorname{rank} Q \cdot\exp(-\C_0\eps_1^2 d)-2\exp(-\C_0\eps_2^2 d)\\
    &\qquad -2\cdot 9^d \exp \left(-\C_{1} \min \left(\frac{\left(\tr Q + \eps_3 d\norm{Q}_F\right)^2}{4\C^6 \norm{Q}_F^2}, \frac{\tr Q + \eps_3 d\norm{Q}_F}{2\C^3 \norm Q} \right)\right), 
  \end{align*}
  we have 
  \begin{align}
    &\left|x^\T Qx - x^\T P^\T PQP^\T Px\right| \le 3\eps_1 \norm{x}^2 \norm{Q}_F, \label{eq:xQx_xPPQ_posPPx_eq_1}\\[5pt]
    &\|P x\|^{2} \leq(1+\eps_2)\|x\|^{2},\label{eq:xQx_xPPQ_posPPx_eq_2} \\[5pt]
    &\left|\min(0, \lambda_{\min}(PQP^\T)) \right| \le \eps_3 \norm{Q}_F.  \label{eq:xQx_xPPQ_posPPx_eq_3}
  \end{align}

  First, we decompose the error $\left|x^\T Qx - x^\T P^\T \mathcal F^+\left(PQP^\T\right) Px\right|$ into two terms:
  \begin{align*}
    &\left|x^\T Qx - x^\T P^\T \mathcal F^+\left(PQP^\T\right) Px\right|\\
    \le &\left|x^\T Qx - x^\T P^\T PQP^\T Px\right| + \left|x^\T P^\T PQP^\T Px-x^\T P^\T \mathcal F^+\left(PQP^\T\right) Px\right|.
  \end{align*}
  The upper bound on the first term is given by \cref{eq:xQx_xPPQ_posPPx_eq_1}. 
  To bound the second term, we define $\bar Q^- = PQP^\T-\mathcal F^+\left(PQP^\T\right)$. Since $-\bar Q^- \succeq O$, 
  we can define its non-negative square root $\sqrt {-\bar Q^-}$. With these notations, 
  we get the upper bound on the second term as follows:
  \begin{align*}
    \left|x^\T P^\T PQP^\T Px-x^\T P^\T \mathcal F^+\left(PQP^\T\right) Px\right|
    =&\left|x^\T P^\T(-\bar Q^-) Px\right|\\
    =&\norm{\sqrt {-\bar Q^-}Px}^2\\
    \le&\norm{\sqrt{-\bar Q^-}}^2 \norm{Px}^2\\
    =&  \norm{\bar Q^-} \norm{Px}^2\\
    =&  \left|\min\bigl(0, \lambda_{\min}(PQP^\T)\bigr)\right| \norm{Px}^2\\
    \le& \eps_3 \norm{Q}_F (1+\eps_2)\|x\|^{2}. \\
    &\quad \quad (\text{by } \cref{eq:xQx_xPPQ_posPPx_eq_2}\cref{eq:xQx_xPPQ_posPPx_eq_3})
  \end{align*}
\end{proof}

In the next lemma we evaluate the probability shown in \cref{cor:neg_spec_of_PQP} or  \cref{thm:xQx_xPPQ_posPPx}.

\begin{lemma}\label{lem:prob_bound}
  Let $Q$ be an $n\times n$ symmetric matrix that satisfies $\tr Q >0$ and define
  $$\tilde r \equiv\frac{\norm{Q}_F^2}{\norm{Q}^2}, \quad \tilde k \equiv \frac{\tr Q}{\norm{Q}}.$$
  Furthermore let $\mathcal{D}\ge \mathcal{C}$.

  If (i) $d \ge -\dfrac{2}{\eps_3}\dfrac{\tilde k}{\sqrt{\tilde r}} + \dfrac{12\mathcal{D}^6}{\C_1\eps_3^2}-\dfrac{\log \delta}{3}$ and 
  (ii) $d < \dfrac{2\mathcal{D}^3\tilde r - \tilde k}{\eps_3 \sqrt{\tilde r}}$, then 
  $$2\cdot 9^d \exp \left(-\C_{1} \min \left(\frac{\left(\tr Q + \eps_3 d\norm{Q}_F\right)^2}{4\C^6 \norm{Q}_F^2}, \frac{\tr Q + \eps_3 d\norm{Q}_F}{2\C^3 \norm Q} \right)\right) \le \delta.$$
\end{lemma} 

\begin{proof}
  For simplicity, we use 
  $p=\dfrac{\tilde k}{\sqrt{\tilde r}}$ and $q = \sqrt{\tilde r}$. 
  The conditions (i) and (ii) are equivalent to 
  \begin{align}
    &d \ge -\dfrac{2p}{\eps_3} + \dfrac{12\mathcal{D}^6}{\C_1\eps_3^2}-\dfrac{\log \delta}{3},\label{eq:prob_bound_eq_1}\\
    &d < \dfrac{2\mathcal{D}^3q - p}{\eps_3}. \label{eq:prob_bound_eq_2}
  \end{align}
 We will prove that
 $$2\cdot 9^d \exp \left(-\C_{1} \min \left(\frac{\left(\tr Q + \eps_3 d\norm{Q}_F\right)^2}{4\mathcal{D}^6 \norm{Q}_F^2}, \frac{\tr Q + \eps_3 d\norm{Q}_F}{2\mathcal{D}^3 \norm Q} \right)\right) \le \delta$$
 holds which will end the proof as $\mathcal{D}\ge \C$.
  
  We can easily show that $2\cdot 9^d\le\exp(3d)$ for all $d\in\mathbb N$ and thus, we obtain
  \begin{align*}
       &2\cdot 9^d \exp \left(-\C_{1} \min \left(\frac{\left(\tr Q + \eps_3 d\norm{Q}_F\right)^2}{4\mathcal{D}^6 \norm{Q}_F^2}, \frac{\tr Q + \eps_3 d\norm{Q}_F}{2\mathcal{D}^3 \norm Q} \right)\right)\\
    \le&\exp \left(3d-\C_{1} \min \left(\frac{\left(\tr Q + \eps_3 d\norm{Q}_F\right)^2}{4\mathcal{D}^6 \norm{Q}_F^2}, \frac{\tr Q + \eps_3 d\norm{Q}_F}{2\mathcal{D}^3 \norm Q} \right)\right)\\
    =&  \exp \left(3d-\C_{1} \min \left(\frac{(p + \eps_3 d)^2}{4\mathcal{D}^6}, \frac{pq + \eps_3 d q}{2\mathcal{D}^3} \right)\right).\\
  \end{align*}
  Since \cref{eq:prob_bound_eq_2} is equivalent to $p + \eps_3 d < 2\mathcal{D}^3q$, we have
  $$\min \left(\frac{(p + \eps_3 d)^2}{4\mathcal{D}^6}, \frac{pq + \eps_3 d q}{2\mathcal{D}^3}\right) = \frac{(p + \eps_3 d)^2}{4\mathcal{D}^6}.$$
  Thus, 
  \begin{align*}
    &2\cdot 9^d \exp \left(-\C_{1} \min \left(\frac{\left(\tr Q + \eps_3 d\norm{Q}_F\right)^2}{4\mathcal{D}^6 \norm{Q}_F^2}, \frac{\tr Q + \eps_3 d\norm{Q}_F}{2\mathcal{D}^3 \norm Q} \right)\right) \\
    \le& \exp \left(3d-\frac{\C_{1} }{4\mathcal{D}^6} (p + \eps_3 d)^2\right).
  \end{align*}

  Next, we show that $3d-\frac{\C_{1}}{4\mathcal{D}^6} (p + \eps_3 d)^2 \le \log \delta$,
  which will complete the proof.
  Note that the quadratic equation,
  \begin{equation*}
    3d-\frac{\C_{1} }{4\mathcal{D}^6} (p + \eps_3 d)^2 = \log \delta 
  \end{equation*}
  is equivalent to 
  \begin{equation*}
    \C_1\eps_3^2 d^2 + (2\C_1\eps_3p-12\mathcal{D}^6)d + 4\mathcal{D}^6 \log \delta + C_1p^2=0
  \end{equation*}
  of which real solutions are given by (if there are any)
  \begin{align*}
    d &=\frac{-2\C_1\eps_3 p + 12\mathcal{D}^6 \pm \sqrt{(2\C_1\eps_3 p - 12\mathcal{D}^6)^2 - 4\C_1\eps_3^2(4\mathcal{D}^6\log \delta+\C_1p^2)}}{2\C_1\eps_3^2}\\
    &= \frac{-\C_1\eps_3 p + 6\mathcal{D}^6 \pm \sqrt{36\mathcal{D}^{12}-12\mathcal{D}^6\C_1\eps_3p-4\mathcal{D}^6\C_1\eps_3^2 \log \delta}}{\C_1\eps_3^2}.
  \end{align*}
  
  If there are no real roots, then $3d-\frac{\C_{1}}{4\mathcal{D}^6} (p + \eps_3 d)^2 < \log \delta$ holds
  for all $d$. Thus, it is sufficient to show that 
  $$d\ge \frac{-\C_1\eps_3 p + 6\mathcal{D}^6 + \sqrt{36\mathcal{D}^{12}-12\mathcal{D}^6\C_1\eps_3p-4\mathcal{D}^6\C_1\eps_3^2 \log \delta}}{\C_1\eps_3^2}.$$

  To show this inequality, we use the following inequality $a-\dfrac{b}{2a}\ge \sqrt{a^2-b}\ (a,b>0, a^2>b)$ which can be easily verified by squaring both sides.
  Applying this inequality with
  $a = 6\mathcal{D}^6$ and $b = 12\mathcal{D}^6\C_1\eps_3p+4\mathcal{D}^6\C_1\eps_3^2 \log \delta$, we obtain
  \begin{equation}\label{eq:prob_bound_eq_3}
    6\mathcal{D}^6 - \frac{12\mathcal{D}^6\C_1\eps_3p+4\mathcal{D}^6\C_1\eps_3^2 \log \delta}{2\cdot 6\mathcal{D}^6} \ge 
    \sqrt{36\mathcal{D}^{12}-12\mathcal{D}^6\C_1\eps_3p-4\mathcal{D}^6\C_1\eps_3^2 \log \delta}, 
  \end{equation}
  which completes the proof:
  \begin{align*}
    d 
    \ge& -\dfrac{2p}{\eps_3} + \dfrac{12\mathcal{D}^6}{\C_1\eps_3^2}-\dfrac{\log \delta}{3} & (\text{by } \cref{eq:prob_bound_eq_1}) \\
    =&  \frac{-\C_1\eps_3 p + 6\mathcal{D}^6}{\C_1\eps_3^2} + \frac{1}{\C_1\eps_3^2} \left(6\mathcal{D}^6 - \frac{12\mathcal{D}^6\C_1\eps_3p+4\mathcal{D}^6\C_1\eps_3^2 \log \delta}{2\cdot 6\mathcal{D}^6}\right)\\
    \ge&\frac{-\C_1\eps_3 p + 6\mathcal{D}^6 + \sqrt{36\mathcal{D}^{12}-12\mathcal{D}^6\C_1\eps_3p-4\mathcal{D}^6\C_1\eps_3^2 \log \delta}}{\C_1\eps_3^2}. &(\text{by } \cref{eq:prob_bound_eq_3})
  \end{align*}
\end{proof}

\begin{remark}\label{rem:quantities_rk}
  The quantities $\tilde r$ and $\tilde k$ defined in \cref{lem:prob_bound} are known as stable rank (also called numerical rank)
  \cite{koltchinskii2017concentration} and effective rank (also called intrinsic dimension) \cite{tropp2015introduction, vershynin2012introduction}. Clearly, $\tilde k, \tilde r \le \operatorname{rank} Q$ and 
  they can be interpreted as the robust version of the usual rank. These quantities are used in covariance estimation~\cite{vershynin2018high}.

\end{remark}
\subsection{The error bound in a special case}

We now evaluate the error between the optimal value of the original problem \cref{eq:problem_P} and that of the convexified projected problem \cref{eq:problem_CRP} under the following  assumptions.
\begin{assumption}\label{assumption_1}
  In the original problem \cref{eq:problem_P}, we assume the followings:
  \begin{enumerate}
    \renewcommand{\labelenumi}{\rm (A\arabic{enumi})}
    \item The problem \cref{eq:problem_P} has a finite optimal value at $x= x^*$.
    \item All the rows of A are unit vectors {\normalfont (i.e., $\norm{A_i}=1$)}.
    \item $\tr Q > 0.$  
    \item There exists a closed ball $B(0, r)$ which is contained in the polytope $Ax\le b$.
  \end{enumerate}
\end{assumption}

(A2) holds without loss of generality: 
if $\norm{A_i} \neq 1$, we replace $A_i$ by $A_i/\norm{A_i}$ and $b_i$ by $b_i/\norm{A_i}$ and then the assumption is satisfied.
(A3) is essential in this paper, though it is replaced by a weaker assumption later in \cref{sec:scaling_preconditioning}. As shown in  \cref{lem:concentration_of_PQP} and \cref{cor:neg_spec_of_PQP}, 
the eigenvalues of $\bar Q = PQP^\T$ concentrate around $\frac{\tr Q}{d}$, and ignoring the negative spectrum of $PQP^\T$
does not change the problem so much especially when $\tr Q$ is large enough.
(A4) also can be weaken later in  \cref{sec:subGenCase}; An essential requirement is that the polytope is full dimensional;
equality constraints are not acceptable in \cref{eq:problem_P}.

We investigate the relationship between $\P$, $\CRP$, and the following problem:
\begin{equation*}
  \CRP_\eps \equiv \min \{u^\T \bar Q^+ u + \bar c^\T u\mid \bar Au \le b + \eps\norm{x^*}\bm 1\},
\end{equation*}
where $\eps>0$ and $\bm 1$ is the all-one vector.

For an optimization problem ${\mathbf F}$, we denote by $\opt{\mathbf F}$ the optimal objective function value of {\bf F}.
We also let $S, T$ and $T_\eps$ be the feasible regions of $\P$, $\CRP$ and $\CRP_\eps$, respectively. 


We can easily show that $\opt{\CRP} \ge \opt{\P}$ and $\opt{\CRP} \ge \opt{\CRP_\eps}$.

\begin{lemma}\label{lem:CRP_is_relaxetion_of_P}
Under \cref{assumption_1} (A1) and (A4), $\opt{\CRP}$ is finite and
  $$\opt{\CRP} \ge \opt{\P}.$$
\end{lemma}

\begin{proof}
  First, we show that $\opt{\CRP}$ is finite.
  \cref{assumption_1} (A4) implies that $u=0$ is feasible for $\CRP$, and therefore,  $\opt{\CRP}<\infty.$ 
  We can confirm that $\opt{\CRP}>-\infty$ by contradiction. If $\opt{\CRP}=-\infty$, 
  there exists a sequence $\{u^k\}_{k=1}^\infty$ in $T$ such that ${u^k}^\T \bar Q^+ u^k + \bar c^\T u^k \to -\infty \quad (k\to \infty)$. 
  Since $u^k\in T$, we have $\bar Au^k = A P^\T u^k \le b$, which implies $P^\T u_k \in S$ and
  \begin{align*}
    \opt{\P}
    \le& (P^\T u^k)^\T Q(P^\T u^k) + c^\T (P^\T u^k)\\
    \le& {u^k}^\T \F^+(PQP^\T) u^k + (Pc)^\T u^k\\
    =& {u^k}^\T \bar Q^+ u^k + \bar c^\T u^k \to -\infty,
  \end{align*}
  which implies a contradiction to \cref{assumption_1} (A1).

  To show the rest part of the proof, let $u^*$ be an optimum of $\CRP$. From the same argument before, we have that $P^\T u^* \in S$ and
  \begin{align*}
    \opt{\P}
    \le& (P^\T u^*)^\T Q(P^\T u^*) + c^\T (P^\T u^*)\\
    \le & {u^*}^\T \bar Q^+ u^* + \bar c^\T u^* \\
    =& \opt{\CRP}.
  \end{align*}
\end{proof}

\begin{lemma}\label{lem:CRPeps_is_relaxetion_of_CRP}
Under \cref{assumption_1} (A1) and (A4), 
  \begin{equation*}
    \opt{\CRP} \ge \opt{\CRP_\eps}.
  \end{equation*}
\end{lemma}  
\begin{proof}
  This follows immediately from $T\subseteq T_\eps$ of two feasible regions. 
  Note that, as shown in \cref{lem:CRP_is_relaxetion_of_P}, we have $0\in T$ and the feasibilities of $\CRP$ and $\CRP_\eps$ are guaranteed.
\end{proof}
  The inequality $\opt{\CRP} \ge \opt{\CRP_\eps}$ includes the case where  
  $\opt{\CRP_\eps}=-\infty$, though, as we will see later in \cref{thm:gap_CRP_CRPeps}, this case does not occur.

Next, we investigate the gap between $\P$ and $\CRP_\eps$.

\begin{theorem}\label{thm:gap_P_CRPeps}
  Let $x^*\in\R^n$ be an optimum of $\P$. Under \cref{assumption_1} (A1), (A2) and (A3),
  for any $\eps, \eps_1, \eps_2, \eps_3, \eps_4,  >0$, we have  
  $$\opt{\P} \ge \opt{\CRP_\eps} - (3\eps_1 + \eps_3 + \eps_2\eps_3 )\norm{x^*}^2 \norm{Q}_F - \eps_4\norm{x^*}\norm{c}$$
  with probability at least 
  \begin{multline*}
    1 - 4m\exp (-\C_0\eps^2 d) - 8\operatorname{rank} Q \cdot\exp(-\C_0\eps_1^2 d)-2\exp(-\C_0\eps_2^2 d)-2\exp(-\C_0\eps_4^2 d) \\
    -2\cdot 9^d \exp \left(-\C_{1} \min \left(\frac{\left(\tr Q + \eps_3 d\norm{Q}_F\right)^2}{4\C^6 \norm{Q}_F^2}, \frac{\tr Q + \eps_3 d\norm{Q}_F}{2\C^3 \norm Q} \right)\right).
  \end{multline*}
\end{theorem}
\begin{proof}
  By  \cref{lem:basic_properties_of_random_matrix}(i)(ii) and  \cref{thm:xQx_xPPQ_posPPx}, with probability at least
  \begin{multline*}
    1 - 4m\exp (-\C_0\eps^2 d) - 8\operatorname{rank} Q \cdot\exp(-\C_0\eps_1^2 d)-2\exp(-\C_0\eps_2^2 d)-2\exp(-\C_0\eps_4^2 d) \\
    -2\cdot 9^d \exp \left(-\C_{1} \min \left(\frac{\left(\tr Q + \eps_3 d\norm{Q}_F\right)^2}{4\C^6 \norm{Q}_F^2}, \frac{\tr Q + \eps_3 d\norm{Q}_F}{2\C^3 \norm Q} \right)\right),
  \end{multline*}
  the following inequalities hold: 
  \begin{align}
  &AP^\T P x^* \le Ax^*+\eps \norm{x^*} \bm 1, \label{eq:APPxstar} \\
  &|{x^*}^\T Qx^* - {x^*}^\T P^\T \mathcal F^+\left(PQP^\T\right) Px^*| 
  \le (3\eps_1 + \eps_3 + \eps_2\eps_3 )\norm{x^*}^2 \norm{Q}_F, \label{eq:xstarQxstar}\\
  &c^\T P^\T P x^* \le c^\T x^*+\eps_4 \norm{x^*}\norm{c}. \label{eq:cPPxstar}
  \end{align}

  \cref{eq:APPxstar} implies that $Px^*$ is a feasible solution for $\CRP_\eps$.
  By \cref{eq:xstarQxstar} and \cref{eq:cPPxstar}, we have
  \begin{align*}
  &\opt{\P}\\
  =&   {x^*}^\T Qx^* + c^\T x^* \\
  \ge& {x^*}^\T P^\T \mathcal F^+\left(PQP^\T\right) Px^* - (3\eps_1 + \eps_3 + \eps_2\eps_3 )\norm{x^*}^2 \norm{Q}_F  + c^\T P^\T Px^*-\eps_4\norm{x^*}\norm{c}\\
  =  &   (Px^*)^\T \bar Q^+ (Px^*) + \bar c^\T (Px^*)- (3\eps_1 + \eps_3 + \eps_2\eps_3 )\norm{x^*}^2 \norm{Q}_F-\eps_4\norm{x^*}\norm{c}\\
  \ge& \opt{\CRP_\eps}  - (3\eps_1 + \eps_3 + \eps_2\eps_3 )\norm{x^*}^2 \norm{Q}_F-\eps_4\norm{x^*}\norm{c}.
  \end{align*}
\end{proof}

It should be noted that 
the error estimate in  \cref{thm:gap_P_CRPeps} includes the information on the optimum of $\P$ (more precisely,
$\|x^*\|$). The estimate seems unrealistic because $x^*$ is not available. However, the necessary quantity is  $\|x^*\|$.
In the case of a bounded feasible region for $\P$,
it may be bounded by some threshold $R$ such as $\|x^*\| \leq R$ by proving that
the feasible region of $\P$ is included in a ball $\|x\| \leq R$.

\begin{theorem}\label{thm:gap_CRP_CRPeps}
Under \cref{assumption_1} (A1), (A2) and (A4),
$\opt{\CRP_\eps}$ is finite and
  we have 
  \begin{equation*}
    \opt{\CRP} \ge \opt{\CRP_\eps} \ge \left(1+\eps\frac{\norm{x^*}}{r}\right)\opt{\CRP}.
  \end{equation*}
\end{theorem}

\begin{proof}
  First, we admit the finiteness of $\opt{\CRP_\eps}$ and show the inequality.
  We will confirm the finiteness at the end of this proof.

  Since we have already shown that $\opt{\CRP} \ge \opt{\CRP_\eps}$ in  \cref{lem:CRPeps_is_relaxetion_of_CRP}, the second inequality:
  \begin{equation*}
    \opt{\CRP_\eps} \ge \frac{\opt{\CRP}}{\alpha},
  \end{equation*}
  where we let $\ds\alpha = 1-\frac{\eps\norm{x^*}}{r+ \eps\norm{x^*}}$, is proven now. 
  Since $B(0, r)\subset S$, $Ax \le b$ for all $x$ satisfying $\norm{x} \le r$. Therefore, 
  \begin{equation}\label{eq:b_le_r}
    b_i \ge \max_{\norm{x}\le r} (Ax)_i =r\norm{A_i} = r\quad(i=1,2,\dots,m).
  \end{equation}
  Next, we construct a feasible solution for $\CRP$ close to $u^*_\eps$, 
  an optimum of $\CRP_\eps$. 
  This will enable us to evaluate the error between $\opt{\CRP}$ and $\opt{\CRP_\eps}$.

  By the definition of $\CRP_\eps$, we have $\bar A u^*_\eps\le b+\eps\norm{x^*}\bm 1$, and thus, for each $i= 1,2,\dots,m$, we have
  \begin{align*}
    \bigl(\alpha\bar A u^*_\eps\bigr)_i &\le \left(1-\frac{\eps\norm{x^*}}{r+ \eps\norm{x^*}}\right)(b_i+\eps\norm{x^*})\\
    &= b_i+\eps\norm{x^*} - \frac{\eps\norm{x^*}}{r+ \eps\norm{x^*}}(b_i+\eps\norm{x^*})\\
    &\le b_i+\eps\norm{x^*} - \frac{\eps\norm{x^*}}{r+ \eps\norm{x^*}}(r+\eps\norm{x^*}) & \text{(by \cref{eq:b_le_r})}\\
    &= b_i.
  \end{align*}
  This implies that $\alpha u^*_\eps$ is feasible for $\CRP$, hence
  \begin{equation}\label{eq:vRPgu'}
      \opt{\CRP} \le g(\alpha u^*_\eps), 
  \end{equation}
  where $g(\cdot)$ is the objective function of $\CRP$ and $\CRP_\eps$, i.e. $g(u) = u^\T\bar Q^+u + \bar c^\T u$. 
  Note that  $\bar Q^+$ is positive semidefinite, which implies $g$ is convex
  and thus  we have
  \begin{equation}\label{eq:gJensen}
      g(\alpha u^*_\eps) = g(\alpha u^* _\eps + (1-\alpha)0) \le \alpha g(u^* _\eps) + (1-\alpha)g(0)=\alpha \opt{\CRP_\eps}.
  \end{equation}

  From the two inequalities \cref{eq:vRPgu'} and \cref{eq:gJensen}, 
  $$\opt{\CRP_\eps} \ge \frac{\opt{\CRP}}{\alpha}$$
  is derived.

  And lastly, we show that $\opt{\CRP_\eps}$ is finite.
  Let $u_\eps$ be an arbitrary feasible solution of $\CRP_\eps$, i.e. $\bar A u_\eps\le b+\eps\norm{x^*}\bm 1$,
  then, substituting $u_\eps$ for $u_\eps^*$ in the above discussion, we have
\[    \opt{\CRP} \le g(\alpha u_\eps) \le  \alpha g(u_\eps).\]
  Since we have already shown the finiteness of $\opt{\CRP}$ in \cref{lem:CRP_is_relaxetion_of_P}, 
  $g(u_\eps)$ is turned out to be lower bounded, and thus $\opt{\CRP_\eps}$ is finite.
\end{proof}

\begin{lemma}\label{lem:mainlem}
Under  \cref{assumption_1},
  for any $\eps,\eps_1, \eps_2, \eps_3, \eps_4 >0$, 
  we have 
  \begin{align*}
    \opt{\CRP} &\ge \opt{\P} \\
    &\ge \left(1+\eps\frac{\norm{x^*}}{r}\right)\opt{\CRP} - (3\eps_1 + \eps_3 + \eps_2\eps_3 )\norm{x^*}^2 \norm{Q}_F - \eps_4\norm{x^*}\norm{c}
  \end{align*}
  with probability at least
  \begin{multline*}
    1 - 4m\exp (-\C_0\eps^2 d) - 8\operatorname{rank} Q \cdot\exp(-\C_0\eps_1^2 d)-2\exp(-\C_0\eps_2^2 d)-2\exp(-\C_0\eps_4^2 d) \\
    -2\cdot 9^d \exp \left(-\C_{1} \min \left(\frac{\left(\tr Q + \eps_3 d\norm{Q}_F\right)^2}{4\C^6 \norm{Q}_F^2}, \frac{\tr Q + \eps_3 d\norm{Q}_F}{2\C^3 \norm Q} \right)\right).
  \end{multline*}
\end{lemma}

\begin{proof}
  This statement follows from  \cref{lem:CRP_is_relaxetion_of_P}, \cref{thm:gap_P_CRPeps} and 
  \cref{thm:gap_CRP_CRPeps}.
\end{proof}

\begin{theorem}[Approximation theorem under \cref{assumption_1}]\label{thm:main_theorem}
  Define
  $\ds\tilde r \equiv\frac{\norm{Q}_F^2}{\norm{Q}^2},  \tilde k \equiv \frac{\tr Q}{\norm{Q}}.$
  Let 
  $\eps, \eps_1, \eps_3, \eps_4 \in (0,1], \delta_1, \delta_2 \in (0,1/2)$ and $\mathcal{D}\ge \C$. 
  Suppose that 
  \begin{enumerate}
    \setlength{\leftskip}{3.0em}
    \renewcommand{\labelenumi}{(\roman{enumi})} 
    \item $d\ge \max\left\{\dfrac{\log ((24\operatorname{rank} Q +6)/\delta_1)}{\C_{0}\eps_1^2}, 
                          \dfrac{\log (6/\delta_1)}{\C_{0}\eps_4^2}, 
                          \dfrac{\log (12m/\delta_1)}{\C_{0}\eps^2}\right\} $, 
    \item $d \ge -\dfrac{2}{\eps_3}\dfrac{\tilde k}{\sqrt{\tilde r}} + \dfrac{12\mathcal{D}^6}{\C_1\eps_3^2}-\dfrac{\log \delta_2}{3}$,
    \item $d < \dfrac{2\mathcal{D}^3\tilde r - \tilde k}{\eps_3 \sqrt{\tilde r}}$.
  \end{enumerate}

  Then with probability at least $1-\delta_1-\delta_2$, the following inequalities hold:
  \begin{align*}
    \opt{\CRP} &\ge \opt{\P} \\
    &\ge \left(1+\eps\frac{\norm{x^*}}{r}\right)\opt{\CRP} - (3\eps_1 + 2\eps_3)\norm{x^*}^2 \norm{Q}_F - \eps_4\norm{x^*}\norm{c}.
  \end{align*}
\end{theorem}

\begin{proof}
  Let $\eps_2=1$  in  \cref{lem:mainlem}, then we have 
  \begin{align*}
    \opt{\CRP} &\ge \opt{\P} \\
    &\ge \left(1+\eps\frac{\norm{x^*}}{r}\right)\opt{\CRP} - (3\eps_1 + 2\eps_3)\norm{x^*}^2 \norm{Q}_F - \eps_4\norm{x^*}\norm{c}
  \end{align*}
  with probability at least
  \begin{multline*}
    1 - 4m\exp (-\C_0\eps^2 d) - 8\operatorname{rank} Q \cdot\exp(-\C_0\eps_1^2 d)-2\exp(-\C_0 d)-2\exp(-\C_0\eps_4^2 d) \\
    -2\cdot 9^d \exp \left(-\C_{1} \min \left(\frac{\left(\tr Q + \eps_3 d\norm{Q}_F\right)^2}{4\C^6 \norm{Q}_F^2}, \frac{\tr Q + \eps_3 d\norm{Q}_F}{2\C^3 \norm Q} \right)\right).
  \end{multline*}

  By (i) and the assumption $\eps_1\le 1$, 
  $$8\operatorname{rank} Q \cdot \exp(-\C_0\eps_1^2 d) + 2\exp(-\C_0 d) 
  \le (8\operatorname{rank} Q +2)\exp(-\C_0\eps_1^2 d) \le \frac{\delta_1}{3}.$$

  (i) also implies 
  \begin{align*}
    2\exp(-\C_0\eps_4^2 d)  \le \frac{\delta_1}{3},\\
    4m\exp (-\C_0\eps^2 d) \le \frac{\delta_1}{3}.
  \end{align*}

  Since (ii), (iii) are the same assumptions as the ones made in  \cref{lem:prob_bound}, we have 
  $$ 2\cdot 9^d \exp \left(-\C_{1} \min \left(\frac{\left(\tr Q + \eps_3 d\norm{Q}_F\right)^2}{4\C^6 \norm{Q}_F^2}, \frac{\tr Q + \eps_3 d\norm{Q}_F}{2\C^3 \norm Q} \right)\right) \le \delta_2,$$
  ending the proof.
\end{proof}

  \cref{thm:main_theorem} has shown some upper and lower bounds on the size $d$
  of a randomly projected QP.
  Indeed, to bound $| \opt{\P}-\opt{\CRP}|$, the existence of these bounds is reasonable.
  Larger $d$ makes the gap between $ \opt{\P}$ and $\opt{\RP}$ small, while smaller $d$ makes the gap between $\opt{\RP}$ and
  $\opt{\CRP}$ small  because the matrix $\bar Q$ ($\approx \frac{\tr Q}{d}I_d$) in the objective of $\opt{\RP}$ tends to be positive definite.
  Therefore, to make the bound of $| \opt{\P}-\opt{\CRP}|$ small, a well-balanced $d$ for both gaps is needed.
  In \cref{examp_d}, we show  that for a certain class of non-convex QPs \cref{eq:problem_P}, there exists $d$ satisfying the above (i)-(iii), and furthermore, how small $d$ can be in those cases.

  Now we discuss 
  how to obtain a feasible solution for the original problem $\P$ from $u^*$, the optimal solution of $\CRP$.
  As shown in the proof of \cref{lem:CRP_is_relaxetion_of_P}, $x'=P ^\T u^*$ is feasible for $\P$
  and 
  \begin{equation*}
    \opt{\CRP} \ge {x'}^\T Qx' + c^\T x' \ge \opt{\P},
  \end{equation*}
  so that $x'$ is a feasible solution of $\P$ that we expect to achieve an approximate optimal value. 

\subsection{The error bound in a more general case} \label{sec:subGenCase}
Previously, 
we assumed that the feasible region contains a sphere centered at the origin ((A4) in \cref{assumption_1}).
Next, we will consider a more general situation, i.e. we make the
following assumption.

\begin{assumption}\label{assumption_2} 
  {\rm (A1),(A2),(A3)} in \cref{assumption_1} and 
  \begin{enumerate}
    \setcounter{enumi}{3}
    \renewcommand{\labelenumi}{\rm (A\arabic{enumi}')}
    \item There exists a closed ball $B(x_0, r)$ which is contained in the polytope $Ax\le b$.
  \end{enumerate}
\end{assumption}

Considering the variable translation $y=x-x_0$, we obtain the translated problem:
\begin{equation*}
  \P_{\mathrm T} \equiv \min_y \{y^\T Qy + (2Qx_0+c)^\T y \mid Ay \le b-Ax_0\}.
\end{equation*}
It is obvious that $\opt{\P_{\mathrm T}} = \opt{\P}-x_0^\T Qx_0-c^\T x_0$, hence it is enough to solve $\P_{\mathrm T}$ instead of $\P$.
Moreover, there exists a closed ball $B(0, r)$ which is contained in the polytope $Ax\le b-Ax_0$ so that we can apply the previous argument.
Define the convexified randomly projected problem of $\P_{\mathrm T}$:
\begin{equation*}
  \CRP_{\mathrm T} \equiv \min_v \{v^\T \bar Q^+ v + (P(2Qx_0+c))^\T v\mid \bar Av \le b -Ax_0\},
\end{equation*}
and then, by  \cref{thm:main_theorem}, we obtain a generalized approximation theorem. 
Note that one of optimal points of $\P_{\mathrm T}$ is $y^*=x^*-x_0$.

\begin{theorem}[Approximation theorem under \cref{assumption_2}]\label{thm:generalized_main_theorem}
  Define
  $\ds\tilde r \equiv\frac{\norm{Q}_F^2}{\norm{Q}^2},  \tilde k \equiv \frac{\tr Q}{\norm{Q}}.$
  Let 
  $\eps, \eps_1, \eps_3, \eps_4 \in (0,1], \delta_1, \delta_2 \in (0,1/2)$ and $\mathcal{D}\ge \C$.

  Suppose that 
  \begin{enumerate}
    \setlength{\leftskip}{3.0em}
    \renewcommand{\labelenumi}{(\roman{enumi})} 
    \item $d\ge \max\left\{\dfrac{\log ((24\operatorname{rank} Q +6)/\delta_1)}{\C_{0}\eps_1^2}, 
                          \dfrac{\log (6/\delta_1)}{\C_{0}\eps_4^2}, 
                          \dfrac{\log (12m/\delta_1)}{\C_{0}\eps^2}\right\} $, 
    \item $d \ge -\dfrac{2}{\eps_3}\dfrac{\tilde k}{\sqrt{\tilde r}} + \dfrac{12\mathcal{D}^6}{\C_1\eps_3^2}-\dfrac{\log \delta_2}{3}$,
    \item $d < \dfrac{2\mathcal{D}^3\tilde r - \tilde k}{\eps_3 \sqrt{\tilde r}}$.
  \end{enumerate}

  Then with probability at least $1-\delta_1-\delta_2$, the following inequality holds:
  \begin{multline*}
    \opt{\CRP_{\mathrm T}} \ge \opt{\P_{\mathrm T}} \\
    \ge \left(1+\eps\frac{\norm{y^*}}{r}\right)\opt{\CRP_{\mathrm T}} - (3\eps_1 + 2\eps_3 )\norm{y^*}^2\norm{Q}_F \\
    - \eps_4\norm{y^*}\norm{2Qx_0+c}.
  \end{multline*}
\end{theorem}

Now we show some conditions on the size $n,m$ and $\cond(Q)$ for non-convex QP \cref{eq:problem_P}
to guarantee the existence of $d$ satisfying the above (i)-(iii) in
\cref{thm:generalized_main_theorem}, where we define $\cond(Q)\equiv \sigma_n/\sigma_1$, the condition number of $Q$. 
In the following proposition, $P(n)$ may be formed
with specific functions such as $\log^j n$ and $n^\tau (0<\tau<1)$. 
\begin{proposition} \label{examp_d}
  Assume that $m<\tilde C n$ holds for some constant $\tilde\C$ and a function $P(n)$ satisfies that
  \begin{equation*}
    \log n \ll P(n) \quad \mbox{ and } \quad P(n) + \sqrt n \ll \frac{\sqrt{nP(n)}}{\cond(Q)^2}.
  \end{equation*} 
  Then, if $n$ is large enough, we can chose
	\begin{equation*}
    d=O\left(\frac{P(n)}{\varepsilon_0^2}\right)
  \end{equation*} 
  such that the 
(i)-(iii) of \cref{thm:generalized_main_theorem} are satisfied. Here, $\varepsilon_0$ is a constant that only depends on $\varepsilon, \varepsilon_1, \varepsilon_4,\delta_1$. 
\end{proposition}
\begin{proof}
  By $(i)$ and the assumption $\log n \ll P(n)$, we can take $d=\frac{P(n)}{\varepsilon_0^2}$
using 
a constant $\varepsilon_0$ that only depends on $\varepsilon, \varepsilon_1, \varepsilon_4,\delta_1$.

  We can chose $\mathcal{D}^6=\mathcal{C}'\left(d+\frac{\tilde{k}}{\sqrt{\tilde{r}}}\right)=\mathcal{C}'\left(\frac{P(n)}{\varepsilon_0^2}+\frac{\tilde{k}}{\sqrt{\tilde{r}}}\right)$, where $\mathcal{C}'$ is a constant, that depends only on $\varepsilon_3,\delta_2$, such that $(ii)$ is satisfied.

  To finish the proof we need to verify (iii): 
  \begin{equation*}
    \frac{P(n)}{\varepsilon_0^2} < \dfrac{2\sqrt{\mathcal{C}'\left(\frac{P(n)}{\varepsilon_0^2}+\frac{\tilde{k}}{\sqrt{\tilde{r}}}\right)}\tilde r - \tilde k}{\eps_3 \sqrt{\tilde r}},
  \end{equation*}
  which is equivalent to:
	\begin{equation}\label{eq:1}
	  \frac{\varepsilon_3}{\varepsilon_0^2}\sqrt{\tilde{r}}P(n)+\tilde{k} < 2\sqrt{\mathcal{C'}}\sqrt{\frac{P(n)}{\varepsilon_0^2}\tilde{r}^2+\tilde{k}\tilde{r}^{3/2}}.
	\end{equation}
	From the definitions of $\tilde{r}$ and $\tilde{k}$, we easily see that
	\begin{equation}\label{eq:2}
	  \frac{1}{\cond(Q)^2}n \le \tilde{r}\le n \quad \mbox{ and } \quad 0<\tilde{k} < n.
  \end{equation}
  Hence, the left-hand-side, $\frac{\varepsilon_3}{\varepsilon_0^2}\sqrt{\tilde{r}}P(n)+\tilde{k}$, of \cref{eq:1} has 
	\begin{equation*}
    \frac{\varepsilon_3}{\varepsilon_0^2} \sqrt{n} P(n) + n
  \end{equation*} 
	as a an upper bound, for $n$ large enough, and the right-hand-side of \cref{eq:1} is lower-bounded by 
	\begin{equation*}
    \frac{2\sqrt{C'}}{\varepsilon_0\cond(Q)^2}n\sqrt{P(n)}.
  \end{equation*} 
  Hence the condition 
  \begin{equation*}
    P(n) + \sqrt n \ll \frac{\sqrt{nP(n)}}{\cond(Q)^2} 
  \end{equation*}
	is enough to prove $(iii)$ and hence that $d=\frac{P(n)}{\varepsilon_0^2}$ satisfies the three condition of \cref{thm:generalized_main_theorem}, for $n$ large enough.  
\end{proof}


\subsection{Relative error}
The approximation inequality shown so far has an additive form. 
But we do not know how large or small the error term appearing in the theorem is compared to the optimal value $\opt{\P_{\mathrm T}}$.
The purpose of this section is to transform the approximation inequality in \cref{thm:generalized_main_theorem} into a multiplicative form, $\eta\cdot \opt{\P_{\mathrm T}} \ge \opt{\CRP_{\mathrm T}} \ge \opt{\P_{\mathrm T}}$, where
$\eta \ge 1$ measures how much the error term is relative to the optimal value.
Writing the approximation as above allows to see the parameters of $\P_{\mathrm T}$ that influence the relative error between the two problems. 

For a matrix $M\in \R^{n\times n}$ and a vector $w\in\R^n$, 
we treat $(M,w)$ as an $(n^2+n)$-dimensional vector 
where the inner product of $(M_1,w_1)$ and $(M_2,w_2)$ is defined as follows: 
$$\tliner{(M_1,w_1)}{(M_2,w_2)}_{\R^{n^2+n}} \equiv \tr (M_1^\T M_2) + w_1^\T w_2.$$
\begin{corollary} \label{multiplicative_err}
  We can rewrite the approximation inequality shown in  \cref{thm:generalized_main_theorem}
into the multiplicative one:
\begin{equation*}
    \left(\frac{r}{r+\eps\norm{y^*}}\right)\left(1+\frac{\sqrt{((3\eps_1 + 2\eps_3)^2+\eps_4^2)}}{\cos \theta^*}\right)\opt{\P_{\mathrm T}} \ge \opt{\CRP_{\mathrm T}} \ge \opt{\P_{\mathrm T}}
\end{equation*}
where $\theta^*$ is the angle between vectors $\xi\equiv (y^*{y^*}^\T, y^*)$ and $\zeta \equiv (Q, 2Qx_0+c) \in \R^{n^2+n}$. 
\end{corollary}

\begin{proof}
By the definition of $\xi$ and $\zeta$, we have,
$$\opt{\P_{\mathrm T}} = \tliner{\xi}{\zeta}_{\R^{n^2+n}}.$$
We also have
\begin{align*}
  \norm{\xi}^2\norm{\zeta}^2 &= (\|y^*{y^*}^\T\|_F^2+\norm{y^*}^2)(\norm{Q}_F^2+\norm{2Qx_0+c}^2)\\
  &=   (\norm{y^*}^4+\norm{y^*}^2)(\norm{Q}_F^2+\norm{2Qx_0+c}^2)\\
  &\ge  \norm{y^*}^4\norm{Q}_F^2 + \norm{y^*}^2\norm{2Qx_0+c}^2,
\end{align*}
and now we can evaluate the error term in  \cref{thm:generalized_main_theorem}:
\begin{align*}
  E &\equiv  (3\eps_1 + 2\eps_3 )\norm{y^*}^2\norm{Q}_F + \eps_4\norm{y^*}\norm{2Qx_0+c}\\
    &\le \sqrt{((3\eps_1 + 2\eps_3)^2+\eps_4^2)(\norm{y^*}^4\norm{Q}_F^2 + \norm{y^*}^2\norm{2Qx_0+c}^2)}\\
    &\phantom{(3\eps_1 + 2\eps_3 )\norm{y^*}^2\norm{Q}_F + \eps_4\norm{y^*}\norm{2Qx_0+c}}(\text{by Schwarz's inequality})\\
    &\le \sqrt{((3\eps_1 + 2\eps_3)^2+\eps_4^2)}\norm{\xi}\norm{\zeta}\\
    &= \sqrt{((3\eps_1 + 2\eps_3)^2+\eps_4^2)} \frac{\tliner{\xi}{\zeta}_{\R^{n^2+n}}}{\cos \theta^*} \\
    &= \sqrt{((3\eps_1 + 2\eps_3)^2+\eps_4^2)} \frac{\opt{\P_{\mathrm T}}}{\cos \theta^*}.
\end{align*}
\end{proof}

%% file: scaling_preconditioning.tex
\section{Scaling and Preconditioning}\label{sec:scaling_preconditioning}

In this section we will provide an error bound for $\CRP$ under a weaker assumption
than (A3) in \cref{assumption_1}.

\begin{assumption}\label{assumption_3} 
    {\rm (A1),(A2)} in \cref{assumption_1}, {\rm (A4')} in \cref{assumption_2} and 
    \begin{enumerate}
      \setcounter{enumi}{2}
      \renewcommand{\labelenumi}{\rm (A\arabic{enumi}')}
      \item At least one of the eigenvalues of Q is positive.
    \end{enumerate}
\end{assumption}

We consider the scaling $y = Uz$ and the following scaled problem:
\begin{equation*}
    \P_{\mathrm T}' \equiv \min_z \{z^\T Q'z + (U^\T (2Qx_0+c))^\T z \mid A'z \le b-Ax_0\}.
\end{equation*}
where $U$ is a scaling invertible matrix, $Q'= U^\T QU$ and $A'=AU$. 
Obviously, we have $\opt{\P_{\mathrm T}'} = \opt{\P_{\mathrm T}}.$ The corresponding convexified randomly projected problem becomes
\begin{equation*}
  \CRP'_{\mathrm T} \equiv \min_w \{w^\T \bar {Q'}^+ w + (PU^\T(2Qx_0+c))^\T w\mid \bar A'w \le b -Ax_0\},
\end{equation*}
and $\opt{\CRP'_{\mathrm T}} = \opt{\CRP_{\mathrm T}}$ holds.

In order to apply the arguments so far, we have to make sure that $\tr Q'$ is positive.
Von Neumann's trace inequality \cite{mirsky1975trace} implies that
\begin{equation}
  \tr Q' = \tr(U^\T QU) = \tr(QUU^\T) \le \sum_{i=1}^n \lambda_i\sigma_i^2,
    \label{eq:trace_ineq}
\end{equation}
where $\lambda_1\ge\lambda_2\ge\dots\ge\lambda_n$ are eigenvalues of $Q$ and $\sigma_1\ge\sigma_2\ge\dots\ge\sigma_n>0$ are singular values of $U$. The equality holds for \cref{eq:trace_ineq} when $U=V^\T \diag(\sigma_1, \sigma_2, \dots, \sigma_n)V$, where the eigenvalue decomposition of $Q$ is given by $Q=V^\T \diag(\lambda_1,\lambda_2,\dots,\lambda_n) V$.
Thus, if $\sigma_1$ and $\sigma_n$ are fixed, the maximum value of $\tr Q'$ is given by
$$\max_{\sigma_2,\dots,\sigma_{n-1}} \tr Q' = 
  \sum_{i=1}^l \lambda_i\sigma_1^2 + \sum_{i=l+1}^n \lambda_i\sigma_n^2,$$
where $l$ is the index determined by $\lambda_1\ge\dots\ge\lambda_l\ge 0 > \lambda_{l+1}\ge\dots\ge\lambda_n$.
Therefore, if $Q$ has at least one positive eigenvalue, we can generate $Q'$ so as to satisfy $\tr Q' > 0$, that
makes it possible to use the same error-bound analysis to  $\P_{\mathrm T}'$. 

Based on the above discussion, we consider the case where the form of $U$ is given as
$$U=V^\T \diag(\sigma_1, \dots, \sigma_1, \sigma_n, \dots, \sigma_n)V \quad \text{($l$ $\sigma_1$s and $n-l$ $\sigma_n$s)}.$$

In the previous theorems on approximation errors of $\CRP$,
$\tilde r$ and $\tilde k$ are used as discussed in \cref{rem:quantities_rk}. In the following theorem on the approximation error,
we will use $\tilde r$ and $\cond(U)$.

\begin{lemma}\label{lem:ball_scaling}
    There exists a closed ball $B\left(0, \dfrac{r}{\norm U}\right)$ which is contained in the polytope $A'z \le b-Ax_0$.
\end{lemma}
\begin{proof}\ \par
    If $\norm z \le \dfrac{r}{\norm U}$, then $\norm{Uz}\le\norm U \norm z\le r$. Since $B(0,r)\subset \{x\mid Ax\le b-Ax_0\}$, we have $Uz\in\{x\mid Ax\le b-Ax_0\}$ and 
    $A'z = A(Uz) \le b-Ax_0.$
\end{proof}

\begin{theorem}[Approximation theorem under \cref{assumption_3}]\label{thm:scaled_main_theorem}
    Define
    $\ds\tilde r \equiv\frac{\norm{Q}_F^2}{\norm{Q}^2}.$
    Let 
    $\eps, \eps_1, \eps_3, \eps_4 \in (0,1], \delta_1, \delta_2 \in (0,1/2)$.
  
    Suppose that 
    \begin{enumerate}
      \setlength{\leftskip}{3.0em}
      \renewcommand{\labelenumi}{(\roman{enumi})} 
      \item $d\ge \max\left\{\dfrac{\log ((24\operatorname{rank} Q +6)/\delta_1)}{\C_{0}\eps_1^2}, 
                            \dfrac{\log (6/\delta_1)}{\C_{0}\eps_4^2}, 
                            \dfrac{\log (12m/\delta_1)}{\C_{0}\eps^2}\right\} $, 
      \item $d \ge \dfrac{12\C^6}{\C_1\eps_3^2}-\dfrac{\log \delta_2}{3}$,
      \item $\ds d < \dfrac{2\C^3\tilde r-\cond(U)^4 \tilde k}{\eps_3\cond(U)^2\sqrt{\tilde r}}$.
    \end{enumerate}
  
    Then with probability at least $1-\delta_1-\delta_2$, the following inequality holds
    \begin{multline}
    \opt{\CRP'_{\mathrm T}} \ge \opt{\P_{\mathrm T}} 
    = \opt{\P'_{\mathrm T}} \\
    \ge \left(1+\eps \cond(U)\frac{\norm{y^*}}{r}\right) \opt{\CRP'_{\mathrm T}} 
    - (3\eps_1 + 2\eps_3)\cond(U)^2\norm{y^*}^2\norm{Q}_F \\- \eps_4\cond(U)\norm{y^*}\norm{2Qx_0+c}.
    \end{multline}
\end{theorem}

\begin{proof}
    Define
    $\ds\tilde r' \equiv\frac{\norm{Q'}_F^2}{\norm{Q'}^2},  \tilde k' \equiv \frac{\tr Q'}{\norm{Q'}}, \Lambda = \diag(\lambda_1,\lambda_2,\dots,\lambda_n)$, and \\ $\Sigma = \diag(\sigma_1,\sigma_2,\dots,\sigma_n)$.
    First, we observe that
    \begin{align*}
      &\norm{Q'} = \norm{U^\T QU} = \norm{V^\T \Sigma \Lambda \Sigma V} = \norm{\Sigma \Lambda \Sigma } \le \sigma_1^2\norm{\Lambda} = \sigma_1^2\norm{Q},\\
      &\norm{Q'}_F = \norm{U^\T QU}_F = \norm{V^\T \Sigma \Lambda \Sigma V}_F = \norm{\Sigma \Lambda \Sigma }_F \ge \sigma_n^2\norm{\Lambda}_F = \sigma_n^2 \norm{Q}_F,\\
      &\tr Q' = \tr (U^\T QU) = \tr (V^\T \Sigma \Lambda \Sigma V) = \tr (\Sigma \Lambda \Sigma)  \le \sigma_1^2\tr \Lambda = \sigma_1^2 \tr Q.
    \end{align*}
    Then, on the condition (iii), 
    we have
    \begin{align*}
          \dfrac{2\C^3\tilde r' - \tilde k'}{\eps_3 \sqrt{\tilde r'}} 
        &=\frac{1}{\eps_3}\left(2\C^3\frac{\norm{Q'}_F}{\norm{Q'}} - \frac{\tr Q'}{\norm{Q'}_F}\right) \\
        &\ge \frac{1}{\eps_3}\left(2\C^3\frac{\norm{Q}_F}{\cond(U)^2\norm{Q}} - \cond(U)^2\frac{\tr Q}{\norm{Q}_F}\right)\\
        &=\dfrac{1}{\eps_3\cond(U)^2}\left(2\C^3\sqrt{\tilde r}-\cond(U)^4\frac{\tilde k}{\sqrt{\tilde r}}\right)\\
        &=\dfrac{2\C^3\tilde r-\cond(U)^4 \tilde k}{\eps_3\cond(U)^2\sqrt{\tilde r}}.
    \end{align*}
    On the other hand, for the condition (ii), it is easy to see that
    $$ \dfrac{12\C^6}{\C_1\eps_3^2}-\dfrac{\log \delta_2}{3} \ge -\dfrac{2}{\eps_3}\dfrac{\tilde k'}{\sqrt{\tilde r'}} + \dfrac{12\C^6}{\C_1\eps_3^2}-\dfrac{\log \delta_2}{3}.$$
    Thus, under the conditions (i)-(iii), we have 
    \begin{enumerate}
        \setlength{\leftskip}{3.0em}
        \renewcommand{\labelenumi}{(\roman{enumi}')} 
        \item $d\ge \max\left\{\dfrac{\log ((24\operatorname{rank} Q +6)/\delta_1)}{\C_{0}\eps_1^2}, 
                              \dfrac{\log (6/\delta_1)}{\C_{0}\eps_4^2}, 
                              \dfrac{\log (12m/\delta_1)}{\C_{0}\eps^2}\right\} $, 
        \item $d \ge -\dfrac{2}{\eps_3}\dfrac{\tilde k'}{\sqrt{\tilde r'}} + \dfrac{12\C^6}{\C_1\eps_3^2}-\dfrac{\log \delta_2}{3}$,
        \item $d < \dfrac{2\C^3\tilde r' - \tilde k'}{\eps_3 \sqrt{\tilde r'}},$
    \end{enumerate}
    which are the same to conditions in  \cref{thm:generalized_main_theorem}.
    By \cref{thm:generalized_main_theorem} and \cref{lem:ball_scaling}, we have, with probability at least $1-\delta_1-\delta_2$, 
    \begin{multline}\label{eq:errorbound_scaling}
    \opt{\CRP'_{\mathrm T}} \ge \opt{\P'_{\mathrm T}}\\
    \ge \left(1+\eps \frac{\norm{z^*}}{r/\norm{U}}\right) \opt{\CRP'_{\mathrm T}}  - (3\eps_1 + 2\eps_3 )\norm{z^*}^2\norm{Q'}_F\\
    - \eps_4\norm{z^*}\norm{U^\T(2Qx_0+c)}.
    \end{multline}
    Note that $\norm{z^*}=\norm{U^{-1}y^*}\le\norm{U^{-1}}\norm{y^*}$, which leads to
    \begin{align*}
        &\frac{\norm{z^*}}{r/\norm{U}} \le \norm{U}\norm{U^{-1}} \frac{\norm{y^*}}{r} = \cond(U)\frac{\norm{y^*}}{r},\\
        &\norm{z^*}^2\norm{Q'}_F \le \norm{U^{-1}}^2\norm{U}^2 \norm{y^*}^2\norm{Q}_F = \cond(U)^2\norm{y^*}^2\norm{Q}_F,\\
        &\norm{z^*}\norm{U^\T(2Qx_0+c)} \le \norm{U^{-1}}\norm{U}\norm{y^*}\norm{2Qx_0+c}=\cond(U)\norm{y^*}\norm{2Qx_0+c}.
    \end{align*}
    By using these inequalities into \cref{eq:errorbound_scaling},
    we obtain an error bound in the claim.
\end{proof}

We can derive an approximation error of $\CRP'_{\mathrm T}$ in a multiplicative form similar to \cref{multiplicative_err}
under \cref{assumption_3}, though we omit the description.

%% file: numerical_experiments.tex
\section{Numerical Experiments}\label{sec:numerical_experiments}

\subsection{Randomly generated problems}
We perform some experiments on randomly generated non-convex QPs. 
In the previous sections, we discussed the error between the optimal values of
the original problem $\P$ and of the convexified projected problem $\CRP$. In 
this section, to estimate the error in practice, 
we compare $\opt{\P}, \opt{\RP}$ and $\opt{\CRP}$ 
so that the errors by random projection and convexification can be verified separately.

Unfortunately, it is difficult to find global optimal solutions of
$\P$ and $\RP$ because they are non-convex QPs. Therefore, 
we use D.C. algorithms \cite{tao2005dc} with a multi-start strategy with 10 randomly chosen initial points
to find a best possible approximated optimal value.
Thus, in this section,  
$\opt{\P}$ and $\opt{\RP}$ denote best possible approximation values of the true optimal values.



Random instances are generated as follows:
$Q$ is a diagonal matrix whose diagonal entries are first drawn from some distribution independently and next normalized to $\norm{Q}_F=1$.
$c=1/\sqrt n \bm1$, $A_i$ $(i=1,2,\dots,m')$ are random unit vectors and $b_i = \bm1$. 
We also add the constraints $-\bm 1 \le x \le \bm 1$ to ensure the boundedness of the feasible region. The total number of constraints is given by $m=m'+2n$.

\begin{figure}[ht]
  \centering
      \subfigure[$\opt{\RP}-\opt{\P}$]{%
          \includegraphics[clip, width=0.49\columnwidth]{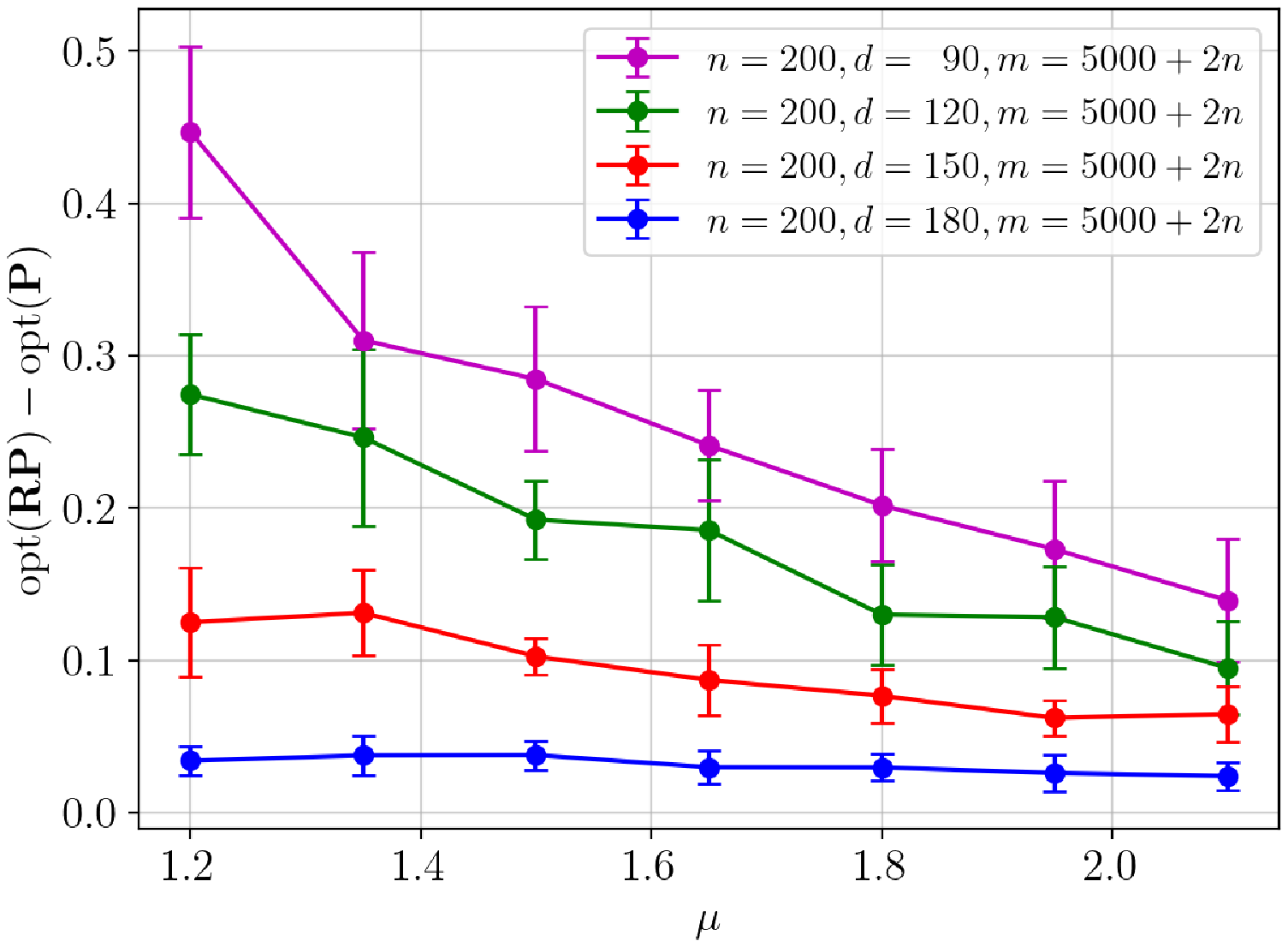}}%
      \subfigure[$\opt{\CRP}-\opt{\RP}$]{%
          \includegraphics[clip, width=0.49\columnwidth]{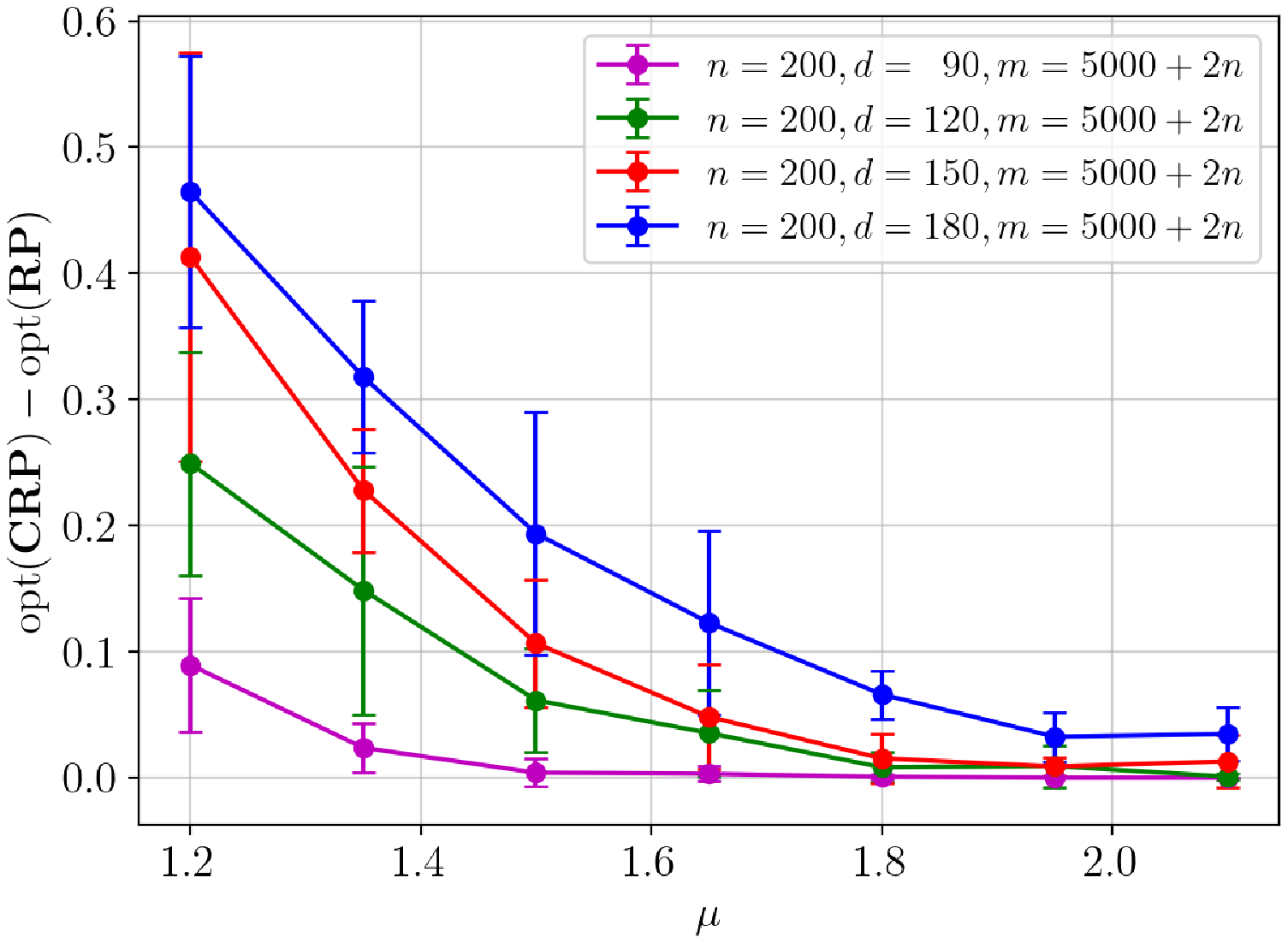}}\\%
      \subfigure[$\opt{\CRP}-\opt{\P}$]{%
          \includegraphics[clip, width=0.49\columnwidth]{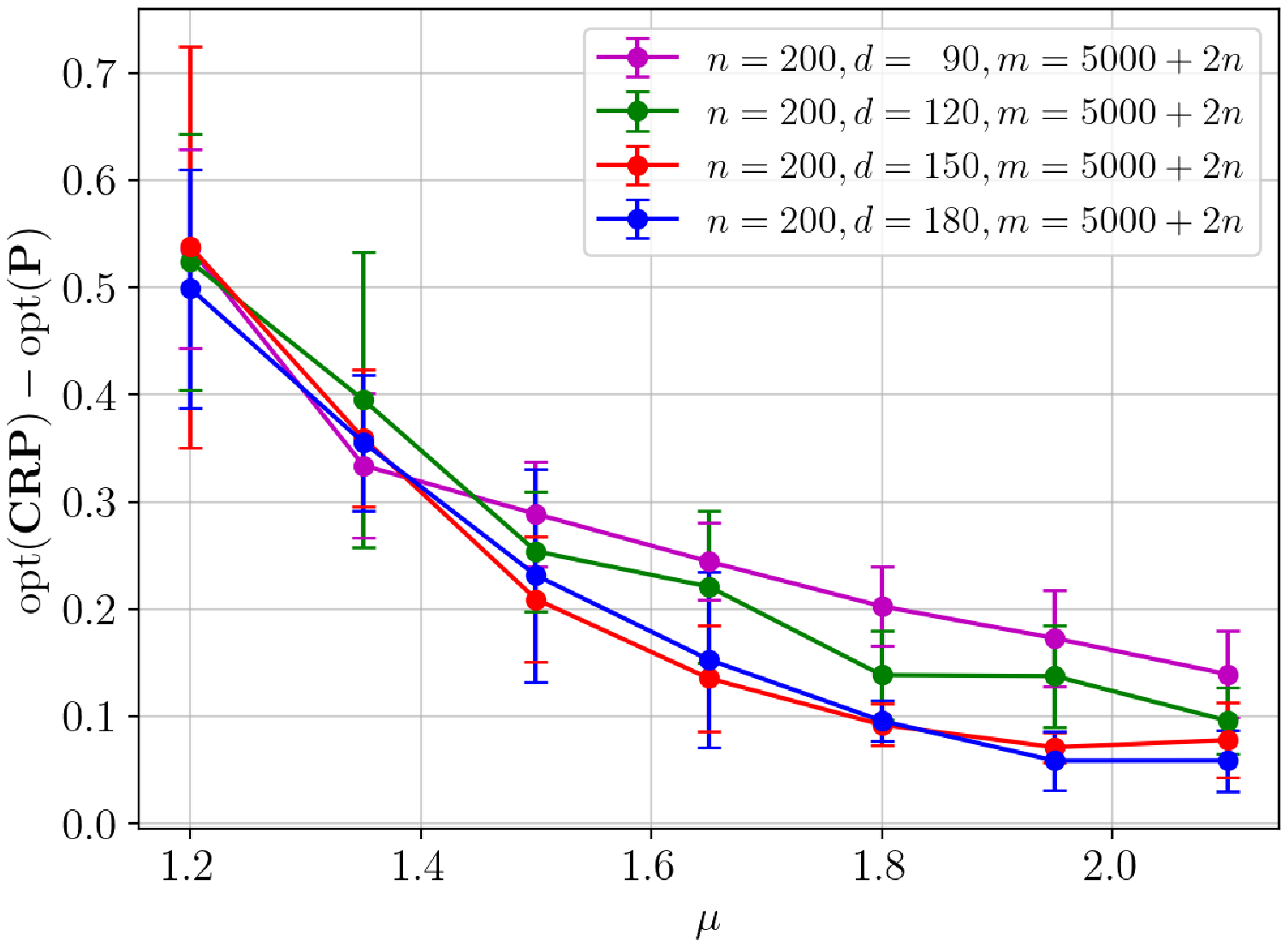}}%
      \caption{The difference between $\opt{\P}, \opt{\RP}$ and $\opt{\CRP}$ versus $\mu$.}
      \label{fig:P_RP_CRP}
\end{figure}

We set $n=200, m'=5000, d=90,120,150,180$ and 
the distribution for randomly chosen diagonal entries of $Q$ is $\mathrm N(\mu,1^2)$,
where $\mu$ is the parameter relating to the convexity of the original problem. 
More precisely, the larger $\mu$ is, the more positively the eigenvalue distribution of $Q$ is skewed.
We calculated optimal values 10 times each for fixed $(d, \mu)$. The average and the standard deviation of the difference between $\opt{\P}, \opt{\RP}$ and $\opt{\CRP}$ 
are shown in \cref{fig:P_RP_CRP}.

\cref{fig:P_RP_CRP}(a) shows $\opt{\RP}-\opt{\P}$, the error due to random projection. 
We confirm that $\opt{\RP}-\opt{\P}$ gets smaller as $d$ gets larger. 
This is because as $d$ gets larger, 
random projections become more likely to preserve geometric quantities or
function values (\cref{lem:basic_properties_of_random_matrix}).

Next we discuss the results shown in \cref{fig:P_RP_CRP}(b), that is the error due to convexification: $\opt{\CRP}-\opt{\RP}$. 
The first observation we can make is that $d$ should be smaller for $\opt{\CRP}-\opt{\RP}$ to be smaller. This is the opposite of the previous observation.
This fact comes from \cref{lem:concentration_of_PQP}. Indeed,
for $\opt{\CRP}-\opt{\RP}$ to be small, $\bar Q^+$ must be a good approximation of $\bar Q$, or equivalently, 
most eigenvalues of $\bar Q$ must be positive, which will be satisfied by setting $d$ small since $\bar Q \approx \frac{\tr Q}{d}I_d$.
We also see that $\opt{\CRP}-\opt{\RP}$ decreases monotonically with respect to $\mu$. 
This is because if $\mu$ is large, then $\P$ and $\RP$ will be nearly convex problems and the error caused by convexification will be small. 

\cref{fig:P_RP_CRP}(c) shows $\opt{\CRP}-\opt{\P}$ we mainly discuss in this paper. 
Since we use convexification, the error between $\opt{\CRP}$ and $\opt{\P}$ 
highly depends on the eigenvalue distribution of $Q$. 
Our method behaves better when the percentage of positive eigenvalues of $Q$ is large.

\subsection{Application for Support Vector Machine Classification with Indefinite Kernels}
\label{subsec:SVM}
Let $K\in\R^{n\times n}$ be a given kernel matrix and $y\in\R^n$ be the vector of labels, with $Y = \diag(y)$.
The classic  soft margin SVM problem  \cite{Cortes1995,lanckriet2004learning} is formulated as:
\begin{equation}\label{eq:SVM_original_problem}
    \min_\alpha \{\alpha^\T YKY \alpha-2\alpha^\T \bm 1  \mid 0\le \alpha\le C\bm 1, \alpha^\T y=0\},
\end{equation}
where $\alpha\in\R^n$ and $C$ is the SVM misclassification penalty which is fixed to $1$ in this paper. 
There are some works (see e.g., \cite{Ong2004,Haasdonk2005,Luss2009,laub2004feature})
investigating applications where kernel matrices formed using similarity measures are not positive semidefinite
and algorithms for SVM \cref{eq:SVM_original_problem} with indefinite $K$.
If $K$ is an indefinite kernel matrix, \cref{eq:SVM_original_problem} is a non-convex QP. 
The corresponding convexified randomly projected problem is
\begin{equation}\label{eq:indefSVM}
    \min_u \{u^\T \F^+(PYKYP^\T) u-2u^\T P\bm 1  \mid 0\le P^\T u\le C\bm 1, u^\T Py=0\},
\end{equation}
where $P\in\R^{d\times n}$ is a random matrix and $u\in\R^{d}$. 
Although we can not apply our theoretical guarantees because the original problem does not have a full dimensional feasible region,
we expect $P^\T u^*$ to be a good approximation of the optimum of the problem \cref{eq:SVM_original_problem}, where $u^*$ is an optimum of the convexified randomly projected problem \cref{eq:indefSVM}.

We performed the experiments on 
the image data of 0, 1 and 7 from the MNIST handwritten digits database \cite{lecun-mnisthandwrittendigit-2010} using
the indefinite simpson score \cite{laub2004feature} as a kernel function value to measure the similarity of two images. 
We experimented with three different binary classifications: 0 and 1, 0 and 7, and 1 and 7.
In all cases, we choose 1000($=n$) train data where each class has 500 points or images
and 400 test data where each class has 200.
The results are shown in \cref{Table:mnist_SVM_0vs1,Table:mnist_SVM_0vs7,Table:mnist_SVM_1vs7}.
We have solved \cref{eq:indefSVM} 20 times with different random $P$ for each $d$ and
 evaluated the optimum of \cref{eq:indefSVM} with test data.
``Training Accuracy'' and ``Test Accuracy'' in the tables refer to the average and standard deviation among 20 training-accuracy and test-accuracy values.
 We confirmed that SVM with 
 simpson score works to find a good approximate solution of the original problem \cref{eq:SVM_original_problem} for appropriate $d$.
We also calculated the accuracy using the optimal solution of the following problem obtained by convexifying \cref{eq:SVM_original_problem} directly:
\begin{equation} \label{accuracyCP}
    \min_\alpha \{\alpha^\T \F^+(YKY) \alpha-2\alpha^\T \bm 1  \mid 0\le \alpha\le C\bm 1, \alpha^\T y=0\},
\end{equation}
and obtained $63.90\%$ training-accuracy and $68.50\%$ test-accuracy, 
$62.00\%$ training-accuracy and $61.00\%$ test-accuracy and
$66.70\%$ training-accuracy and $53.00\%$ test-accuracy for the binary classification of 0 and 1, 0 and 7, and 1 and 7, respectively, 
so that we conclude that 
combining random projections and convexification performs as well or better than just convexification alone.

\begin{table}[p]
    \centering
    \caption{Accuracy (average and standard deviation of 20 times for each $d$) of \cref{eq:indefSVM} for the MNIST 0-1, while \cref{accuracyCP} achieved $63.90\%$ training-accuracy and $68.50\%$ test-accuracy }
      \begin{tabular}{|c|c|c|} \hline
        $d$ & Training Accuracy (\%) & Test Accuracy (\%) \\ \hline \hline
        300 & 48.47 $\pm$ 17.13 & 48.45 $\pm$ 16.89 \\ \hline			
        400 & 62.92 $\pm$ 17.90 & 63.55 $\pm$ 20.13 \\ \hline			
        500 & 80.16 $\pm$ 17.17 & 82.25 $\pm$ 18.53 \\ \hline			
        600 & 95.58 $\pm$ \phantom{0}0.43 & 97.18 $\pm$ \phantom{0}0.40 \\ \hline			
        700 & 97.12 $\pm$ \phantom{0}0.27 & 98.48 $\pm$ \phantom{0}0.46 \\ \hline			
        800 & 97.15 $\pm$ \phantom{0}0.56 & 96.53 $\pm$ \phantom{0}0.43 \\ \hline			
        900 & 88.40 $\pm$ \phantom{0}3.01 & 87.53 $\pm$ \phantom{0}4.12 \\ \hline			
        1000 & 64.60 $\pm$ \phantom{0}1.56 & 68.53 $\pm$ \phantom{0}2.40 \\ \hline
      \end{tabular}
      \label{Table:mnist_SVM_0vs1}
\end{table}

\begin{table}[p]
    \centering
    \caption{Accuracy (average and standard deviation of 20 times for each $d$) of \cref{eq:indefSVM} for the MNIST 0-7, while \cref{accuracyCP} achieved $62.00\%$ training-accuracy and $61.00\%$ test-accuracy }
      \begin{tabular}{|c|c|c|} \hline
        $d$ & Training Accuracy (\%) & Test Accuracy (\%) \\ \hline \hline
        300 & 49.90 $\pm$ \phantom{0}9.24 & 49.40 $\pm$ 10.27 \\ \hline
        400 & 59.67 $\pm$ 14.56 & 61.83 $\pm$ 16.52 \\ \hline
        500 & 76.99 $\pm$ 20.22 & 79.18 $\pm$ 20.56 \\ \hline
        600 & 95.73 $\pm$ \phantom{0}0.29 & 98.43 $\pm$ \phantom{0}0.18 \\ \hline
        700 & 94.25 $\pm$ \phantom{0}0.85 & 95.53 $\pm$ \phantom{0}0.91 \\ \hline
        800 & 84.00 $\pm$ \phantom{0}2.99 & 81.93 $\pm$ \phantom{0}3.95 \\ \hline
        900 & 69.37 $\pm$ \phantom{0}3.42 & 66.15 $\pm$ \phantom{0}2.57 \\ \hline
        1000 & 62.42 $\pm$ \phantom{0}0.96& 61.05 $\pm$ \phantom{0}0.57 \\ \hline
      \end{tabular}
      \label{Table:mnist_SVM_0vs7}
\end{table}

\begin{table}[p]
    \centering
    \caption{Accuracy (average and standard deviation of 20 times for each $d$) of  \cref{eq:indefSVM} for the MNIST 1-7, while \cref{accuracyCP} achieved $66.70\%$ training-accuracy and $53.00\%$ test-accuracy }
      \begin{tabular}{|c|c|c|} \hline
        $d$ & Training Accuracy (\%) & Test Accuracy (\%) \\ \hline \hline
        300 & 55.09 $\pm$ 14.40 & 53.43 $\pm$ 12.80 \\ \hline
        400 & 64.88 $\pm$ 15.64 & 62.80 $\pm$ 14.14 \\ \hline
        500 & 89.00 $\pm$ 12.61 & 83.30 $\pm$ 11.56 \\ \hline
        600 & 95.00 $\pm$ \phantom{0}0.85 & 90.23 $\pm$ \phantom{0}1.54 \\ \hline
        700 & 92.56 $\pm$ \phantom{0}1.60 & 86.53 $\pm$ \phantom{0}2.95 \\ \hline
        800 & 84.26 $\pm$ \phantom{0}3.09 & 76.60 $\pm$ \phantom{0}4.27 \\ \hline
        900 & 72.44 $\pm$ \phantom{0}3.20 & 60.83 $\pm$ \phantom{0}4.19 \\ \hline
        1000 & 66.96 $\pm$ \phantom{0}0.63 & 53.60 $\pm$ \phantom{0}0.78 \\ \hline
      \end{tabular}
      \label{Table:mnist_SVM_1vs7}
\end{table}

%% file: conclusions.tex
\section{Conclusions}\label{sec:conclusions}
Random projections have been applied to solve optimization problems in suitable lower-dimensional spaces  
 in various existing works.
However, to the best of our knowledge, it is the first time they are used to build a convex approximation for a non-convex quadratic optimization problem.
In this paper, we proved that the randomly projected problem $\RP$ that is proposed in \cite{d2019random} is close to a convex problem. This allowed us to propose a convexified randomly projected problem, $\CRP$, that we used to obtain an approximate optimal value of
$\P$. 

In our framework, the existence of a value $d$, that will correspond to the dimension after projections, depends on the distribution of the eigenvalues of $Q$. We proved that even if $\tr{Q}$ is negative then, under some additional error cost, we could use scaling and preconditioning to transform the problem into a new one where the theory applies. 
To confirm that our method is practical,
we applied our framework to SVM classification problem with indefinite kernel, though the problem setting does not satisfy the conditions necessary for the theoretical guarantee. 
As shown in \cref{subsec:SVM}, 
our method is able to find  good approximate global optimal solutions by only solving $\CRP$, which scores as well or better than solving a problem that is only a convexification of the original problem.
At least, it is worth trying our method for a non-convex quadratic problem 
since $\CRP$ is convex and its size is smaller than the original problem and $\opt{\CRP}$ can be obtained by the solver with few computational resources.

One of the directions for future research is to generalize the objective function and constraints, 
which is still difficult since our argument depends on \cref{lem:basic_properties_of_random_matrix} that shows that random projections preserve linear or quadratic function values.
For a general objective function, 
we can consider an iterative method using quadratic approximation of the function at each point,
but obtaining theoretical guarantees in such a case needs further investigations.